\documentclass[11pt]{article}

\usepackage{latexsym,amsmath,amssymb,amsthm}
\usepackage{color,graphicx}
\usepackage{url,hyperref}

\def\R{{\mathbb R}}
\def\C{{\mathbb C}}
\def\bd{\partial}
\def\minus{\backslash}
\DeclareMathOperator{\cl}{Cl}
\newcommand{\arxiv}[1]{\href{http://arxiv.org/abs/#1}{\tt arXiv:#1}}

\newtheorem{theorem}{Theorem}[section]
\newtheorem{lemma}[theorem]{Lemma}
\newtheorem{corollary}[theorem]{Corollary}
\newtheorem{proposition}[theorem]{Proposition}
\newtheorem{definition}[theorem]{Definition}
 
\makeatletter\let\@fnsymbol\@arabic\makeatother

\begin{document}
\title{Polynomials vanishing on Cartesian products: The Elekes-Szab\'o Theorem revisited\let\thefootnote\relax\footnotetext{
Work on this paper by Orit E. Raz and Micha Sharir was supported by Grant 892/13 from the Israel Science Foundation, and
by the Israeli Centers of Research Excellence (I-CORE) program (Center No.~4/11).
Work by Micha Sharir was also supported by Grant 2012/229 from the U.S.--Israel Binational Science Foundation, 
and by the Hermann Minkowski-MINERVA Center for Geometry at Tel Aviv University.
Work on this paper by Frank de Zeeuw was partially supported by Swiss National Science Foundation Grants 200020-144531 and 200021-137574.}}

\author{
Orit E. Raz\thanks{%
School of Computer Science, Tel Aviv University,
Tel Aviv
, Israel.
{\sl oritraz@post.tau.ac.il} }
\and
Micha Sharir\thanks{%
School of Computer Science, Tel Aviv University,
Tel Aviv
, Israel.
{\sl michas@post.tau.ac.il} }
\and
Frank de Zeeuw\thanks{%
Mathematics Department, EPFL, Lausanne, Switzerland.
{\sl fdezeeuw@gmail.com}} }

\date{}
\maketitle

\begin{abstract}
Let $F\in\C[x,y,z]$ be a constant-degree polynomial,
and let $A,B,C\subset\C$ be finite sets of size $n$.
We show that $F$ vanishes on at most $O(n^{11/6})$ points of the 
Cartesian product $A\times B\times C$,
unless $F$ has a special group-related form. 
This improves a theorem of Elekes and Szab\'o \cite{ES12}, and generalizes a result of Raz, Sharir, and Solymosi \cite{RSS14a}. 
The same statement holds over $\R$, and a similar statement holds
when $A, B, C$ have different sizes (with a more involved bound replacing $O(n^{11/6})$).

This result provides a unified tool for improving bounds in various Erd\H os-type problems in combinatorial geometry,
and we discuss several applications of this kind.
\end{abstract}


\section{Introduction}\label{sec:intro}
In 2000, Elekes and R\'onyai \cite{ER00} proved the following result.
Given a constant-degree real polynomial $f(x,y)$, and finite sets $A,B,C\subset \R$ of size $n$, 
we have
\[\big|\left\{(x,y,z)\in\R^3\mid z-f(x,y)=0\right\} \;\cap\; (A\times B\times C)\big|=o(n^2),\]
unless $f$ has one of the forms $f(x,y)=g(h(x)+k(y))$ or $f(x,y)=g(h(x)k(y))$, with univariate real polynomials $g,h,k$.
Recently, Raz, Sharir, and Solymosi \cite{RSS14a} extended an argument introduced in \cite{SSS13} to improve the upper bound  
to $O(n^{11/6})$ (when $f$ does not have one of the special forms).

Elekes and Szab\'o \cite{ES12} generalized the result of \cite{ER00} to any complex algebraic surface of the form
\[Z(F):=\left\{(x,y,z)\in\C^3\mid F(x,y,z)=0\right\},\]
where $F$ is an irreducible polynomial in $\C[x,y,z]$. 
They showed that if $A,B,C\subset\C$ are finite sets of size $n$, then $|Z(F)\cap (A\times B\times C)|$ 
is subquadratic in $n$, unless $F$ has a certain exceptional form. 
The exceptional form of $F$ in this statement is harder to describe (see $(ii)$ in Theorem \ref{thm:main1} below), 
but is related to an underlying group structure that describes the dependencies of $F$ on each of the variables 
(similar to the addition or multiplication that appear in the exceptional forms of $F(x,y,z)=z-f(x,y)$ in \cite{ER00, RSS14a}).
The upper bound that Elekes and Szab\'o obtained, when $F$ is not exceptional, was 
$|Z(F) \cap (A\times B\times C)|=O(n^{2-\eta})$, for an undetermined constant $\eta>0$ that depends on the degree of $F$.

\paragraph{Our results.}
In this paper, we show that the theorem of Elekes and Szab\'o holds for $\eta=1/6$, thereby extending the strengthened result 
of \cite{RSS14a} to the generalized setup in \cite{ES12}. More precisely, our main result is the following theorem.

\begin{theorem}[{\bf Balanced case}]\label{thm:main1}
Let $F\in \C[x,y,z]$ be an irreducible polynomial of degree $d$, and assume that none of the derivatives 
$\partial F/\partial x$, $\partial F/\partial y$, $\partial F/\partial z$ is identically zero. Then one of the following two statements holds.\\
$(i)$ For all $A,B,C\subset \C$ with $|A|=|B|=|C|=n$ we have
\[|Z(F) \cap (A\times B\times C)|=O(d^{13/2}n^{11/6}).\]
$(ii)$ There exists a one-dimensional subvariety
$Z_0\subset Z(F)$,
such that for all $v\in Z(F)\minus Z_0$, 
there exist open sets $D_1,D_2,D_3\subset \C$ and one-to-one analytic functions $\varphi_i: D_i\to \C$ with analytic inverses,
 for $i=1,2,3$,
such that $v\in D_1\times D_2\times D_3$ and for all $(x,y,z)\in D_1\times D_2\times D_3$,
\[(x,y,z)\in Z(F)~~~\text{if and only if}~~~\varphi_1(x)+\varphi_2(y)+\varphi_3(z) = 0.\]
\end{theorem}

Note that for an arbitrary nonzero polynomial $F\in \C[x,y,z]$ of degree $d$ and $A,B,C\subset \C$ with $|A|=|B|=|C|=n$, the Schwartz-Zippel lemma (stated as Lemma \ref{lem:schwartzzippel} in the appendix) gives the bound $|Z(F)\cap(A\times B\times C)| \leq dn^2$.
Thus, for polynomials that do not satisfy property $(ii)$, the bound in property $(i)$ gives an improvement when $d$ is not too large (to be precise, for $d=o(n^{1/33})$).

When property $(ii)$ holds, property $(i)$ fails.
Indeed, consider any $v=(x_0,y_0,z_0)$ and $\varphi_i,D_i$ as in property $(ii)$.
If we set $t_1=\varphi_1(x_0)$, $t_2=\varphi_2(y_0)$, and $t_3=\varphi_3(z_0)$, then we have $t_1+t_2+t_3=0$.
Now choose $A\subset D_1$, $B\subset D_2$, and $C\subset D_3$ so that we have $\varphi_1(A)=\{t_1,t_1+a,t_1+2a,\ldots,t_1+(n-1)a\}$, 
$\varphi_2(B)=\{t_2,t_2+a,t_2+2a,\ldots,t_2+(n-1)a\}$, and 
$\varphi_3(C)=\{t_3,t_3-a,t_3-2a,\ldots, t_3-(n-1)a\}$; this is clearly possible for $a\in \C$ with a sufficiently small absolute value. 
Then $|Z(F)\cap (A\times B\times C)|\geq n^2/4$.

The theorem can easily be extended to polynomials $F$ that are not irreducible. 
Indeed, we can factor $F$ into irreducible factors and apply the theorem to each factor.
Then we can conclude that either $(i)$ holds for $F$, or $(ii)$ holds for one of the factors of $F$,
or one of the factors of $F$ has an identically zero partial derivative, i.e., does not depend on one of the variables.

Our proof also works when the sets $A,B,C$ do not have the same size.
Such an ``unbalanced'' form was not considered in \cite{ER00} or \cite{ES12},
but similar unbalanced bounds were obtained in \cite{RSS14a}, and they are useful in applications where the roles of 
$A,B,C$ are not symmetric (see for instance Corollary \ref{cor:directions}).
We obtain the following result, 
which subsumes Theorem \ref{thm:main1} (as will be argued later); we have stated both for clarity.

\begin{theorem}[{\bf Unbalanced case}]\label{thm:main2}
In Theorem \ref{thm:main1}, property $(i)$ can be replaced by:\\
$(i^*)$ For all triples $A,B,C\subset \C$ of finite sets, we have
\begin{align*}
|Z(F) \cap (A\times B\times C)|=O\Big(\hspace{-2pt}\min\Big\{
 d^{{13}/{2}}|A|^{1/2}|B|^{2/3}|C|^{2/3}&+ d^{17/2}|A|^{1/2}\Big(|A|^{1/2}+|B|+|C|\Big),\\
d^{{13}/{2}}|B|^{1/2}|A|^{2/3}|C|^{2/3}&+d^{{17}/{2}}|B|^{1/2}\Big(|B|^{1/2}+|A|+|C|\Big),\\
d^{13/2}|C|^{1/2}|A|^{2/3}|B|^{2/3}+&d^{17/2}|C|^{1/2}\Big(|C|^{1/2}+|A|+|B|\Big)\Big\}\Big).
\end{align*}
\end{theorem}

We also have the following specialization of Theorem \ref{thm:main2} when $F$ is a real polynomial.
Of course, Theorems \ref{thm:main1} and \ref{thm:main2} also hold when $F$ is real, but it 
does not immediately follow that, in property $(ii)$ of these theorems, the functions $\varphi_i$ 
can be chosen so that they map $\R$ to $\R$.
The following theorem shows that this is indeed the case.
We write $Z_\R(F)$ for the real zero set of a real polynomial, and we refer to \cite{BPR03} for the definition of the dimension of a real zero set.
Note that if $Z_\R(F)$ has dimension less than 2, then it is not hard to obtain a better bound than in $(i)$.

\begin{theorem}[{\bf Real case}]\label{thm:main3}
Let $F\in \R[x,y,z]$ be a polynomial of degree $d$
that is irreducible over $\R$, such that $Z_\R(F)$ has dimension 2.
Then property $(ii)$ in both Theorems \ref{thm:main1} and \ref{thm:main2} can be replaced by:\\
$(ii)_\R$ 
There is a one-dimensional subvariety $Z_0\subset Z_\R(F)$
such that for all $v\in Z_\R(F)\minus Z_0$,
there are open intervals $I_1,I_2,I_3\subset \R$ and one-to-one real-analytic functions\footnote{%
  A real function $f:I\to \R$ on an interval $I$ is \emph{real-analytic} if it has a power series expansion at each point of $I$.}
$\varphi_i: I_i\to \R$ with real-analytic inverses, 
for $i=1,2,3$,
such that $v\in I_1\times I_2\times I_3$ and for all $(x,y,z)\in I_1\times I_2\times I_3$,
\[(x,y,z)\in Z(F)~~~\text{if and only if}~~~\varphi_1(x)+\varphi_2(y)+\varphi_3(z) = 0.\]
\end{theorem} 

\paragraph{Discussion.}
Our proof follows the setup of Sharir, Sheffer, and Solymosi \cite{SSS13}.
We convert the problem at hand into an incidence problem for points and curves, and then apply a Szemer\' edi-Trotter-like incidence theorem (in our case a variant due to Solymosi and De Zeeuw \cite{SZ15}) to obtain the bound in $(i)$.
This theorem does not apply when the constructed curves have many common components, and in this case the special form in $(ii)$ is deduced.
In the various instances of this setup, including Raz, Sharir, and Solymosi \cite{RSS14a}, it is this second part of the analysis that is the most difficult, and the same is true for our proof.  

Although the results in this paper generalize those of \cite{RSS14a}, 
our analysis of the case where the curves have many common components uses entirely different machinery.
Instead of the purely algebraic study of properties of polynomials that was used in \cite{RSS14a}, the approach here requires more advanced tools from algebraic geometry,
and applies them in a considerably more involved style.
This part of our proof was inspired by a technique used by Tao \cite{Tao11a} for a variant of the Elekes-R\'onyai problem over finite fields.

That the current problem is considerably more difficult than the Elekes-R\'onyai problem (in spite of the similarities) can also be seen by comparing the original respective studies in \cite{ER00} and in \cite{ES12}. 
We regard the considerable simplification (on top of the improvement of the bound) of the analysis of Elekes and Szab\'o in \cite{ES12} 
as a major outcome of this paper. 

We note that the polynomial dependence of our bound on the degree of $F$ is also a
significant feature, because it allows us to obtain non-trivial bounds for polynomials of non-constant (but not too large) degree.
In some of the applications mentioned below, an improved dependence on $d$ could have interesting consequences.
In the proof of \cite{ES12}, determining this dependence was out of the question, since their exponent depended on the degree of the polynomial.
Another feature is that we work over $\C$ instead of $\R$, as was partially done in \cite{ES12}, but not in \cite{RSS14a}.

In order to apply our results to a specific problem, like those mentioned below,
one needs to verify that the polynomial $F$ that arises does not have the special form in property $(ii)$. 
This step is far from being straightforward. 
The problem can be reduced to verifying a certain identity in the partial derivatives of $F$ (see \cite[Lemma 33]{ES12}), but even for polynomials of fairly low degree, this can be computationally forbidding. 
In other cases, one may not have an explicit form for $F$, and then a deeper analyis of the specific problem is required.

\paragraph{Applications.}
Besides being an interesting problem in itself, the Elekes-Szab\'o setup arises in many problems in combinatorial geometry.
To demonstrate this, we describe how the setup can be applied to the following problem, which was studied in \cite{ES12,SS}:
Given three \emph{non-collinear} points $p_1, p_2,p_3$ and a point set $P$ in $\R^2$,
obtain a lower bound for the number of distinct distances between $p_1, p_2,p_3$ and the points of $P$.

Let $D$ denote the set of squared distances between $p_1, p_2,p_3$ and the points in $P$.
A point $q\in P$ determines three squared distances to $p_1$, $p_2$, $p_3$, given by
$$
a  = (x_q-x_{p_1})^2 + (y_q-y_{p_1})^2,~~
b  = (x_q-x_{p_2})^2 + (y_q-y_{p_2})^2,~~
c  = (x_q-x_{p_3})^2 + (y_q-y_{p_3})^2 .
$$
The variables $x_q$ and $y_q$ can be eliminated from these equations to yield a quadratic equation $F(a,b,c)=0$
(with coefficients depending on $p_1$, $p_2$, $p_3$).
By construction, for each point $q\in P$, the corresponding squared distances $a$, $b$, $c$ belong to $D$. 
The resulting triples $(a,b,c)$ are all distinct, so $F$ vanishes at $|P|$ triples of $D\times D\times D$. 
Moreover, using the fact that $p_1$, $p_2$, $p_3$ are non-collinear, one can show\footnote{As mentioned in the discussion, this step is not straightforward.} that $F$ does not have the special form in property $(ii)_\R$ of Theorem~\ref{thm:main3}. 
Then property $(i)$ gives $|P| = O(|D|^{11/6})$, or $|D| = \Omega(|P|^{6/11})$, which is the same lower bound that was obtained in \cite{SS} using an ad hoc analysis.
When $p_1$, $p_2$, and $p_3$ are collinear, $F$ becomes
a linear polynomial, in which case it certainly satisfies property $(ii)_\R$, and the above bound on $|D|$ does not hold; in this case the minimum is $\Theta(|P|^{1/2})$ (see \cite{ES12,SS}).

Many other combinatorial questions involving geometric notions such as distances, slopes, or collinearity lead to polynomial relations of the form $F(x,y,z)=0$, and often they can be reduced to studying the number of zeros that such polynomials attain on a Cartesian product.
Several examples where $F$ has the form $z - f(x,y)$ are given in \cite{RSS14a}. 
The following is a sample of other problems that fit into the current framework, but that do not have the explicit form $z-f(x,y)$  required in \cite{RSS14a}.

\medskip

\noindent $(1)$  
Bounding from below the number of distinct distances determined by a point set contained in an algebraic curve~\cite{C14, PZ14}.

\smallskip

\noindent $(2)$ 
Bounding from above the number of triple intersection points of three families of unit circles, 
where the circles in each family pass through a fixed point \cite{ESSz, RSS14b}. 

\smallskip

\noindent $(3)$ 
Bounding from above the number of collinear triples among $n$ points on an algebraic curve \cite{ES13}. 

\medskip

Each of these problems can be seen to fit the current framework, and property $(i)$ of Theorem \ref{thm:main1} 
would give an interesting bound, if one can show that property $(ii)$ does not hold.
For problems $(1)$ and $(2)$, the bound that one would obtain in this way was already achieved 
in the respective references~\cite{PZ14,RSS14b}, using an ad hoc analysis.

That leaves problem $(3)$, for which so far no direct approach has succeeded.
In Section \ref{sec:collinear}, we show how this problem can be attacked with our theorem, using ideas based on those of Elekes and Szab\'o \cite{ES13}.
We state one of the resulting theorems here, but several variants can be found in Section \ref{sec:collinear}.
We say that a triple is \emph{proper} if its entries are distinct.

\begin{theorem}
Any $n$ points on a constant-degree irreducible algebraic curve in $\C^2$ determine $O(n^{11/6})$ proper collinear triples, unless the curve is a line or a cubic.
\end{theorem}

\paragraph{Organization.}
The paper is organized as follows. 
The proof of Theorem~\ref{thm:main2} is given in Section~\ref{sec:proofofmain}, except for the key Proposition~\ref{prop:keyprop}, which is proved in Section~\ref{sec:proofofkey}, and the incidence bound, which is presented in Section~\ref{sec:incgeom}.
Theorem~\ref{thm:main1} is deduced from Theorem~\ref{thm:main2} at the end of Section~\ref{sec:proofofmain}.
Section~\ref{sec:real} shows how to deduce the real case, Theorem~\ref{thm:main3}, from the proof of the complex case.
The applications of our theorem to collinear triples on curves are given in Section~\ref{sec:collinear}. 
The appendix presents some algebraic geometry infrastructure that is heavily used in the proof.

\paragraph{Acknowledgments.}
Part of this research was performed while the authors were visiting the Institute for Pure and Applied Mathematics (IPAM), which is supported by the National Science Foundation. The authors deeply appreciate the stimulating environment and facilities provided by IPAM, which have facilitated the intensive and productive collaboration that have lead to this paper. 
Some of the insights in our analysis were inspired by talks given by Terry Tao at IPAM about his work \cite{Tao11a}.
The authors would also like to thank Kaloyan Slavov, J\'ozsef Solymosi, and Hong Wang for several helpful discussions.


\section{Proof of Theorem \ref{thm:main2}}\label{sec:proofofmain}
In this section we prove Theorem \ref{thm:main2}, except that the proof of the crucial Proposition \ref{prop:keyprop} 
is deferred to  Section \ref{sec:proofofkey}, and the incidence bound that we use is established in Section \ref{sec:incgeom}.
As noted above, the basic tools that we require for the analysis are given in the Appendix, where they are stated, referenced,
and in some cases also proved.

Let $F\in\C[x,y,z]$ be as in the statement of the theorem, and let $A,B,C\subset\C$ be finite sets.
The quantity we wish to bound is
$$
M:=|Z(F)\cap (A\times B\times C)|.
$$
The strategy of the proof is to transform the problem of bounding $M$ into an incidence problem for points and curves in $\C^2$. 
The latter problem can then be tackled using the machinery that we establish in Section~\ref{sec:incgeom}, {\em provided} that 
the resulting curves have well-behaved intersections, in a sense that we will make precise in Definition \ref{def:boundedmult}.
A major component of the proof is to show that 
if the points and curves that we are about to define 
do not have well-behaved intersections in this sense, 
then $Z(F)$ must have the special form described in property $(ii)$ of the theorem.

\subsection{Quadruples}
Define the set of quadruples
$$
Q:=\big\{(b,b',c,c')\in B^2\times C^2\mid \exists a\in A~\text{such that}~F(a,b,c)=F(a,b',c')=0\big\}.
$$
We show, in Lemma~\ref{lem:MtoQ} below, that $M$ can be bounded in terms of $|Q|$. 
The proof requires the following lemma, which says that a surface can contain only a bounded number of vertical lines, unless it is a cylindrical surface (i.e., it consists of all vertical lines through a curve $h(u,v)=0$ in the $uv$-plane).

\begin{lemma}\label{lem:excepa}
Let $H\in\C[u,v,w]$ be an irreducible polynomial of degree $d$ with $\partial H/\partial w$ not identically zero. Then
$$
\displaystyle\Bigl|\Bigl\{(u_0,v_0)\in\C^2\mid H(u_0,v_0,w) =0 ~\text{for all}~w\in \C
\Bigr\}\Bigr|\leq d^2.
$$
\end{lemma}
\begin{proof}
Write 
\[H(u,v,w)=\sum_{i=0}^d \alpha_i(u,v)w^i,\]
with suitable polynomials $\alpha_0(u,v),\ldots,\alpha_d(u,v)$. 
For $(u_0,v_0)\in \C^2$ we have $H(u_0,v_0,w)\equiv 0$ if and only if $\alpha_i(u_0,v_0) =0$ for all $i$.

Observe that $\alpha_0\equiv 0$ would contradict the irreducibility of $H$, 
and $H \equiv \alpha_0$ would contradict $\partial H/\partial w$ not being identically zero.
In addition, the polynomials $\alpha_i$ cannot have a common factor, since this would again contradict the irreducibility of $H$. 
So there exists an index $i_0>0$ such that $\alpha_0$ and $\alpha_{i_0}$ do not have a common factor.
Therefore, by B\'ezout's inequality (Theorem \ref{thm:bezout}),
we have $|Z(\alpha_0)\cap Z(\alpha_{i_0})|\le d^2$.
It follows that there are at most $d^2$ points $(u_0,v_0)$ such that $H(u_0,v_0,w)\equiv 0$.
This proves the lemma.
\end{proof}

\begin{lemma}\label{lem:MtoQ}
We have
$\displaystyle{M=O\left( d^{1/2}|A|^{1/2}|Q|^{1/2}+ d^2|A|\right).}$
\end{lemma}
\begin{proof}
Define 
\[R:= \{(a,b,b',c,c')\in A\times B^2\times C^2\mid F(a,b,c)=F(a,b',c')=0\},
\]
and consider the standard projection $\tau:\C\times \C^4\to \C^4$ (in which the first coordinate is discarded).
We have $Q = \tau(R)$.
For each $a\in A$, we write 
\[
\left(B\times C\right)_a:=\{(b,c)\in B\times C\mid F(a,b,c)=0\},
\]
so that 
\[ |R| = \sum_{a\in A}|\left(B\times C\right)_a|^2.
\]
Using the Cauchy-Schwarz inequality, we have
\[M
= \sum_{a\in A}|\left(B\times C\right)_a|
\le  |A|^{1/2} \Big(\sum_{a\in A}|\left(B\times C\right)_a|^2\Big)^{1/2}
\le 
|A|^{1/2}|R|^{1/2}.
\]

We claim that $|R|\le d|Q|+O(d^4|A|)$. To prove this, let \[
S :=\left\{(b,b',c,c')\in B^2\times C^2\mid  F(a,b,c)= F(a,b',c')=0~\text{for all}~a\in \C
\right\}.\]
A double application of Lemma \ref{lem:excepa} gives $|S|= O(d^4)$.
Observe that for $(b,b',c,c')\in Q\minus S$ we have $|\tau^{-1}(b,b',c,c')\cap R|\leq d$, 
while for $(b,b',c,c')\in S$ we have $|\tau^{-1}(b,b',c,c')\cap R|= |A|$.
Thus 
\[|R| = |\tau^{-1}(Q)| = |\tau^{-1}(Q\minus S)|+|\tau^{-1}(S)| \leq d|Q| + O(d^4|A|),\]
which proves the claim and the lemma.
\end{proof}

In what follows, we derive an upper bound on $|Q|$. 
It will turn out that, when we fail to obtain the bound we are after, 
 $F$ must have the special form in property $(ii)$.

\subsection{Curves}

\paragraph{Primal curves.}
For every point $(y,y')\in \C^2$, we define 
\begin{equation}\label{eq:gamma}
\gamma_{y,y'}:=\cl\Bigl(\big\{(z,z')\in\C^2\mid \exists x\in\C~\text{such that}~F(x,y,z)=F(x,y',z')=0 \big\}\Bigr),
\end{equation}
where $\cl(X)$ stands for the {\em Zariski closure} of $X$ (see Section \ref{sec:projclos}).\footnote{%
  The set before the closure can be viewed as a projection of a variety in $\C^3$.
  Since we work in affine space, the projection of an algebraic variety need not be a variety, and this is why we need to take 
  the Zariski closure in the definition of $\gamma_{y,y'}$. 
  We could instead have worked with constructible sets in affine space, or with projective varieties in projective space, 
  and then the closure would not have been needed, but this does not seem to simplify the proof in any significant way.}
It is not always true that the set $\gamma_{y,y'}$ is a curve; it can turn out to be zero-dimensional or two-dimensional.
The following lemma quantifies the exceptional cases, allowing us to exclude them in what follows.

\begin{lemma}\label{lem:quantelim}
Let $F\in \C[x,y,z]$ be an irreducible polynomial of degree $d$ such that none of $\partial F/\partial x, \partial F/\partial y,\partial F/\partial z$ is identically zero. 
Then there is a finite set $\mathcal{S}\subset \C^2$ with $|\mathcal{S}|=O(d^4)$ such that, for all $(y_1,y_2)\not\in \mathcal{S}$,
the set $\gamma_{y_1,y_2}$, as defined in \eqref{eq:gamma}, is an algebraic curve of degree at most $d^2$, or the empty set.
\end{lemma}
\begin{proof}
Define
\[R := \Bigl\{y_0\mid  \exists x_0 ~\text{such that}~F(x_0,y_0,z)\equiv 0
~\text{or}~
 \exists z_0 ~\text{such that}~F(x,y_0,z_0)\equiv 0 \Bigr\}.\]
Then $|R|=O(d^2)$ by Lemma \ref{lem:excepa}. 
We will show that if $\gamma_{y_1,y_2}$ is not an algebraic curve or an empty set, then $y_1,y_2\in R$.

Consider $\C^3=\C\times\C^2$ with the coordinates $(x,z,z')$, and let $\tau:\C\times\C^2\to \C^2$ be the projection $(x,z,z')\mapsto (z,z')$.
Then 
$\gamma_{y_1,y_2} = \cl(\tau(X))$, where
\[
X := Z(F(x,y_1,z),F(x,y_2,z'))\subset\C^3,
\]
and both $F(x,y_1,z)$, $F(x,y_2,z')$ are considered as polynomials in $x,z,z'$.

Suppose that $\gamma_{y_1,y_2}$ has a zero-dimensional component $\{(z_1,z_2)\}$ (i.e., an isolated point).
Every component of $X$ is at least one-dimensional, as follows from Lemma \ref{lem:dimension} in the appendix, so
$X$ must contain the line $Z(z-z_1,z'-z_2)$, because that is the only one-dimensional variety that is mapped to $\{(z_1,z_2)\}$  by $\tau$.
Thus $F(x,y_1,z_1)\equiv 0$ and $F(x,y_2,z_2)\equiv 0$, 
which means that $y_1,y_2\in R$.

Now suppose that $\gamma_{y_1,y_2}$ has a two-dimensional component (i.e., $\gamma_{y_1,y_2}=\C^2$).
Then, by Lemma \ref{lem:projectionpreserves}, either $X$ has a two-dimensional component, or $X=\C^3$.
The latter would imply that $F(x,y_1,z)\equiv 0$ (as a polynomial in $x,z$), so $y-y_1$ would divide $F$, 
contradicting the irreducibility of $F$.
If $X$ has a two-dimensional component, then $F(x,y_1,z)$ and $F(x,y_2,z')$ (regarded as polynomials in $x,z,z'$) 
have a common factor, and this factor must be a univariate polynomial in $x$. 
In particular, they have a common factor of the form $x-x_0$, 
so that $F(x_0,y_1,z)\equiv 0$ and $F(x_0,y_2,z) \equiv 0$,
which means that $y_1,y_2\in R$.

Put $\mathcal{S}:= R\times R$. We have shown that, when $(y_1,y_2)\not\in \mathcal{S}$,
$X$ and $\gamma_{y_1,y_2}$ are both purely one-dimensional or both empty.
In the former case, since $X$ is defined by two polynomials of degree at most $d$, it has degree at most $d^2$, by Lemma \ref{lem:degreebound}.
Since $\gamma_{y_1,y_2}$ is the closure of the image of $X$ under a projection, it also has degree at most $d^2$, by Lemma \ref{lem:projectionpreserves}.
\end{proof}

Let $\mathcal{S}\subset \C^2$ be the set given by Lemma~\ref{lem:quantelim} for our $F$. That is, $|\mathcal{S}|=O(d^4)$ and, 
for every $(y,y')\in \C^2\minus\mathcal S$, the set $\gamma_{y,y'}$ is an algebraic curve of degree at most $d^2$, or the empty set 
(a possibility we can safely ignore).

\paragraph{Dual curves.}
We define, in an analogous manner, a dual system of curves by switching the roles of the $y$- and $z$-coordinates, as follows. 
For every point $(z,z')\in\C^2$, we define
\[
\gamma^*_{z,z'}:=\cl\Bigl(\big\{(y,y')\in\C^2\mid 
\exists x\in\C~\text{such that}~F(x,y,z)=F(x,y',z')=0\big\}\Bigr).
\]
As above, Lemma \ref{lem:quantelim} implies 
that there exists an exceptional set $\mathcal{T}$ of size $O(d^4)$, such that for every 
$(z,z')\in \C^2\minus \mathcal T$ the set $\gamma_{z,z'}^*$ is an algebraic curve of degree at most $d^2$ (or empty).

We would like to say that ``$(z,z')\in\gamma_{y,y'}$ if and only if $(y,y')\in\gamma_{z,z'}^*$'', but this is not quite true because of the closure operation.
By a basic observation from algebraic geometry (see Lemma~\ref{lem:projectionclosure}),
the closures in the definitions of $\gamma_{y,y'}$ and $\gamma^*_{z,z'}$ add only finitely many points to either curve.
It follows that for all but finitely many points $(z,z')\in\gamma_{y,y'}$ we have $(y,y')\in\gamma_{z,z'}^*$,
and for all but finitely many $(y,y')\in\gamma_{z,z'}^*$ we have $(z,z')\in\gamma_{y,y'}$.

We will analyze what happens when many of these curves have a large common intersection.
The following definition introduces terminology for this, which we will use throughout.

\begin{definition}\label{def:popcurves}
We say that an irreducible algebraic curve $\gamma\subset\C^2$ is a {\em popular curve} if there exist at least $d^4+1$ distinct points $(y,y')\in\C^2\minus \mathcal{S}$ such that $\gamma\subset\gamma_{y,y'}$. 
We denote by $\mathcal C$ the set of all popular curves.\\
Similarly, we say that an irreducible algebraic curve $\gamma^*\subset\C^2$ is {\em a popular dual curve} if there exist at least $d^4+1$ distinct points $(z,z')\in\C^2\minus \mathcal{T}$ such that $\gamma^*\subset\gamma_{z,z'}^*$. We denote by $\mathcal{D}$ the set of all popular dual curves. 
\end{definition}

The main step in our proof is the following key proposition, which shows that we can exclude the popular curves and popular dual curves.
Its proof is quite involved and takes up Section \ref{sec:proofofkey}. 
Note that the statement is only about $F$ and does not involve the specific sets $A,B,C$.

\begin{proposition}\label{prop:keyprop}
Either $F$ satisfies property $(ii)$ of Theorem \ref{thm:main2}, or both of the following properties hold.\\
$(a)$ There exists a one-dimensional variety $\mathcal{Y}\subset\C^2$ of degree $O(d^{11})$ containing $\mathcal{S}$, 
such that for every $(y,y')\in\C^2\minus \mathcal{Y}$, the curve $\gamma_{y,y'}$ does not contain a popular curve.\\
$(b)$ There exists a one-dimensional variety $\mathcal{Z}\subset\C^2$ of degree $O(d^{11})$ containing $\mathcal{T}$,
such that for every $(z,z')\in\C^2\minus \mathcal{Z}$, the dual curve $\gamma_{z,z'}^*$ does not contain a popular dual curve.
\end{proposition}

\subsection{Incidences}
We continue with the analysis, assuming the truth of Proposition~\ref{prop:keyprop}.
We introduce the following set of points and {\em multiset} of curves (some of the curves may coincide as point sets):
\[\Pi:=(C\times C)\minus \mathcal{Z}~~~~~~\text{and}~~~~~~\Gamma := \{\gamma_{b,b'}\mid (b,b')\in(B\times B)\minus\mathcal{Y}\}.\]
By definition, for every $(b,b',c,c')\in Q$, 
we have $(c,c')\in \gamma_{b,b'}$ and $(b,b')\in \gamma^*_{c,c'}$ 
(albeit not necessarily vice versa, 
because the definition of the curves involves a closure, 
and does not require the qualifying point $x$ to be in $A$).
This lets us relate $|Q|$ to $I(\Pi,\Gamma)$, the number of incidences between these points and curves;
since $\Gamma$ is a multiset, these incidences are counted with the multiplicity of the relevant curves.

\begin{lemma}\label{lem:Qtoincs}
We have $\displaystyle{|Q| \le I(\Pi, \Gamma) + O\left(d^{13}|B||C|+d^4|B|^2+d^4|C|^2\right).}$
\end{lemma}
\begin{proof}
Any $(b,b',c,c')$ in $Q$ that is not counted in $I(\Pi,\Gamma)$ must have
$(b,b')\in\mathcal{Y}$ or $(c,c')\in\mathcal{Z}$.

We have $|\mathcal{Y}\cap (B\times B) | = O(d^{11}|B|)$ and 
$|\mathcal{Z}\cap (C\times C)|=O(d^{11}|C|)$ by the Schwartz-Zippel lemma (Lemma \ref{lem:schwartzzippel}).\footnote{To be precise, we apply Lemma \ref{lem:schwartzzippel} to the purely one-dimensional component of $\mathcal{Y}$, which actually has degree $O(d^7)$ (see Lemma \ref{lem:X2}), and add the number of zero-dimensional components of $\mathcal{Y}$, which is $O(d^{11})$. The overestimate in the weaker argument given above does not affect the subsequent bounds. We do the same for $\mathcal{Z}$.}
For any $(b,b')\in\mathcal{Y}\minus \mathcal{S}$, we have $|\gamma_{b,b'}\cap (C\times C)| =O(d^2|C|)$, and for any $(c,c')\in \mathcal{Z}\minus \mathcal{T}$ we have $|\gamma_{c,c'}^*\cap (B\times B)| =O(d^2|B|)$, by the Schwartz-Zippel lemma again.
Thus the contribution from these excluded pairs is $O(d^{13}|B||C|)$.

For $(b,b')\in \mathcal{S}$, $(c,c')\in \mathcal{T}$, the sets $\gamma_{b,b'}$ and $\gamma_{c,c'}^*$ could in fact be 
all of $\C^2$, so we only have the trivial bounds $|\gamma_{b,b'}\cap (C\times C)| =O(|C|^2)$ and $|\gamma_{c,c'}^*\cap (B\times B)| =O(|B|^2)$.
Fortunately, we have $|\mathcal{S}| =O(d^4)$ and $|\mathcal{T}| =O(d^4)$,
so the contribution from these pairs is only $O(d^4|B|^2+d^4|C|^2)$.
\end{proof}

We now define exactly in what sense we require the curves to have well-behaved intersections.

\begin{definition}\label{def:boundedmult}
Let $\Pi$ be a finite set of distinct points in $\C^2$, 
and let $\Gamma$ be a finite multiset of curves in $\C^2$.
We say that the system $(\Pi,\Gamma)$ has \emph{$(\lambda,\mu)$-bounded multiplicity} 
if\,\footnote{In both (a) and (b) the curves should be counted with their multiplicity.}\\ 
$(a)$ for any curve $\gamma\in \Gamma$,
there are at most $\lambda$ curves $\gamma'\in \Gamma$ such that 
there are more than $\mu$ points contained in both $\gamma$ and $\gamma'$; and\\
$(b)$ for any point $p\in \Pi$, 
there are at most $\lambda$ points $p'\in \Pi$ such that there are more than $\mu$ curves that contain both $p$ and $p'$.
\end{definition}

We claim that the system $(\Pi,\Gamma)$ has $(d^6,d^4)$-bounded multiplicity.
Indeed, by Proposition \ref{prop:keyprop}(a) and the fact that we have avoided $\mathcal{Y}$ when defining $\Gamma$, 
any component of a curve $\gamma\in\Gamma$ is not a popular curve, and is thus shared with at most $d^4$ other curves. 
The curve $\gamma$ has at most $d^2$ irreducible components (see Lemma \ref{lem:degreebound}), so there are at most $d^4\cdot d^2=d^6$ curves $\gamma'\in \Gamma$ such that $\gamma$ and $\gamma'$ have a common component. 
Curves $\gamma'$ that do not have a common component with $\gamma$ intersect it in at most $d^4$ points
by B\'ezout's inequality (Theorem \ref{thm:bezout}); thus condition (a) in the definition of $(d^6,d^4)$-bounded multiplicity is satisfied. 
The argument for condition (b) is fully symmetric.

In Section~\ref{sec:incgeom} we derive an incidence bound, Theorem~\ref{cor:multcor}, resembling the classical Szemer\'edi-Trotter point-line incidence bound \cite{ST83}.
It applies to a set $\Pi$ of points and a multiset $\Gamma$ of curves in $\C^2$ of degree at most $\delta$,
under the conditions that $(\Pi,\Gamma)$ has $(\lambda,\mu)$-bounded multiplicity 
and that $\Pi$ is a subset of a Cartesian product.
The theorem asserts that
\[I(\Pi, \Gamma) 
=O\left(\delta^{4/3}\lambda^{4/3}\mu^{1/3}|\Pi|^{2/3}|\Gamma|^{2/3}+\lambda^2\mu|\Pi|+\delta^4\lambda|\Gamma|\right).
\]
Applying this bound with $|\Pi|\leq |C|^2$ and $|\Gamma|\leq |B|^2$, and with $\delta = d^2$, $\lambda =d^6$, and $\mu = d^4$, we get
\begin{align*}
 I(\Pi,\Gamma) & 
= O\left((d^2)^{4/3}(d^6)^{4/3}(d^4)^{1/3}|B|^{4/3}|C|^{4/3}
+(d^6)^2d^4|B|^2
+(d^2)^4d^6|C|^2\right)\\
& = O\left(d^{12}|B|^{4/3}|C|^{4/3}
+d^{16}|B|^2
+d^{14}|C|^2\right),
\end{align*}
which, together with Lemma \ref{lem:Qtoincs}, gives
\begin{align*}
|Q|&= I(\Pi,\Gamma) + O\left(d^{13}|B||C|+d^4|B|^2+d^4|C|^2\right)\\
 &= O\Big(d^{12}|B|^{4/3}|C|^{4/3}
+d^{16}|B|^2
+d^{14}|C|^2\Big).
\end{align*}
Then, from Lemma \ref{lem:MtoQ}, we get
\begin{align}
M&\leq d^{1/2}|A|^{1/2}|Q|^{1/2}+d^2|A|\nonumber\\
&=
O\left(
d^{13/2}|A|^{1/2}|B|^{2/3}|C|^{2/3}+d^{17/2}|A|^{1/2}|B|+d^{15/2}|A|^{1/2}|C|+d^2|A|\right),\label{eq:unbalbound}
\end{align}
which gives the first of the three bounds in Theorem \ref{thm:main2}$(i)$ (with some overestimate, or rather ``rounding up'',
in the exponents of $d$ in the third and fourth terms).
The other two follow by symmetric arguments.
This completes the proof of Theorem \ref{thm:main2}.

To deduce Theorem \ref{thm:main1} from Theorem \ref{thm:main2}, we set $|A|=|B|=|C|=n$ and observe that for $d=O(n^{1/6})$ the first term of \eqref{eq:unbalbound} dominates, which gives the bound $M=O(d^{13/2}n^{11/6})$ in property $(i)$ of Theorem \ref{thm:main1}.
On the other hand, 
when $d=\Omega(n^{1/6})$, 
then the bound $M=O(d^{13/2}n^{11/6})$ is implied by the Schwartz-Zippel lemma (Lemma \ref{lem:schwartzzippel}), which gives the upper bound $M\leq dn^2$.



\section{Proof of Proposition \ref{prop:keyprop}}\label{sec:proofofkey}
\subsection{Overview of the proof}\label{sec:overview}
We adapt an idea used by Tao \cite{Tao11a} to study the expansion of a polynomial $P(x,y)$ over finite fields.
As part of his analysis he considered the map $\Psi:\mathbb \C^4\to\mathbb \C^4$ defined by
\[
\Psi:(a,b,c,d)\mapsto (P(a,c),P(a,d),P(b,c),P(b,d)),
\]
where $P$ is a polynomial in $\C[x,y]$.
Tao showed that if the image $\Psi(\C^4)$ is four-dimensional, then lower bounds on the expansion of $P$ can be derived.
On the other hand, if the image has dimension at most 3,
then $P$ must have one of the special forms $G(H(x)+K(y))$ or $G(H(x)K(y))$ (as in \cite{ER00}), for suitable
polynomials $G,H,K$. 
Tao proved this by observing that in this case the determinant of the Jacobian matrix of $\Psi$ must vanish identically, leading to an identity for the partial derivatives of $P$, from which the special forms of $P$ can be deduced. 

Following Tao's general scheme, albeit in a different context, we define a variety
\begin{align*}
V:=\big\{(x,x',y,y',z_1, &z_2,z_3,z_4)\in\C^8\mid\\
&F(x,y,z_1)=F(x,y',z_2)=F(x',y,z_3)=
F(x',y',z_4)=0\big\}.
\end{align*}
Note that if $F(x,y,z)=z-P(x,y)$, then $V$ is the \emph{graph} of the map $\Psi$ above.
Also observe that if we fix $y,y'$ in $V$ and eliminate $x,x'$, the range of the last four coordinates of $V$ is $\gamma_{y,y'}\times \gamma_{y,y'}$ 
(up to the closure operation).
For a general polynomial $F$, near most points $v\in V$, we use the implicit function theorem  
to represent $V$ as the graph of a locally defined analytic function
\[
\Phi_v:(x,x',y,y')\mapsto\big( g_1(x,y),~g_2(x,y'),~g_3(x',y),~g_4(x',y') \big).
\]
This function serves as a local analogue of the map $\Psi$ above.
If the determinant of the Jacobian matrix of $\Phi_v$ vanishes on $V$ in some neighborhood of $v$,
then we obtain the special form of $F$. 
This derivation is similar to that of Tao, although our special form requires a somewhat different treatment. 

The other side of our argument, when the determinant of the Jacobian is not identically zero,
is very different from that of Tao.
(We will gloss over many details in the rest of this outline.)
We want to show that there are only finitely many popular curves.
We show that if $\gamma$ is a popular curve 
(i.e., there are more than $d^4$ curves $\gamma_{y,y'}\in \Gamma$ that contain $\gamma$),
then it is \emph{infinitely} popular, in the sense that there is a one-dimensional curve $\gamma^*$ 
of pairs $(y,y')\in \C^2$ for which $\gamma_{y,y'}$ contains $\gamma$.
For $V$, this implies that if we restrict $(y,y')$ to $\gamma^*$ and project to the last four coordinates, then the image is contained 
in $\gamma\times \gamma$.
In other words, the local map $\Phi_v$ sends an open subset of the three-dimensional variety 
$\C^2\times \gamma^*$ to an open subset of the two-dimensional variety $\gamma\times \gamma$. 
The inverse mapping theorem 
now tells us that the determinant of the Jacobian of $\Phi_v$ vanishes on the three-dimensional variety $\C^2\times \gamma^*$. 
Given that by assumption this determinant is not identically zero,
its zero set is three-dimensional, so $\C^2\times \gamma^*$ must be one of its $O_d(1)$ irreducible components.
It follows that there are only $O_d(1)$ popular dual curves, which in turn implies, with some additional reasoning, that there are only $O_d(1)$ popular curves.
This, together with the first part of the argument, essentially establishes Proposition~\ref{prop:keyprop}.

\subsection{The varieties \texorpdfstring{$V$, $V_0$, and $W$}{V, V0 and W}}
\label{sec:thevarietyV}
Throughout Section~\ref{sec:proofofkey} we write 
$\pi_1:\C^8\to\C^4$ and  $\pi_2:\C^8\to\C^4$
for the standard projections onto the first and the last four coordinates, respectively.
We define (as already mentioned above) the variety
\begin{align*}
V:=\big\{(x,x',y,y',z_1, z_2,z_3,z_4)&\in\C^8\mid\\
&F(x,y,z_1)=F(x,y',z_2)=F(x',y,z_3)=
F(x',y',z_4)=0\big\}.
\end{align*}
We first prove that $V$ has the dimension that one expects from a variety defined by four equations in $\C^8$.

\begin{lemma}\label{lem:Vdim4}
The variety $V$ has dimension $4$.
\end{lemma}
\begin{proof}
The variety $V$ is not empty, since it contains the point $(x,x,y,y,z,z,z,z)$ for any point $(x,y,z)\in Z(F)$.
Thus, by Lemma \ref{lem:dimension},
$V$ has dimension at least 4.

With a suitable permutation of the coordinates, we can write the set $V$ as a disjoint union
\begin{equation}\label{eq:union}
V=\bigcup_{(y,y')\in \C^2}\{(y,y')\}\times V_{y,y'}\times  V_{y,y'},
\end{equation}
where
$$
V_{y,y'}:=\{(x,z_1,z_2)\mid F(x,y,z_1)=0,F(x,y',z_2)=0\}\subset \C^3.
$$
Note that, if we project $V_{y,y'}$ to the last two coordinates and take the closure, 
we get the set $\gamma_{y,y'}$.
In the proof of Lemma \ref{lem:quantelim}, we saw that $V_{y,y'}$ is one-dimensional for all but $O(d^4)$ points $(y,y')\in \C^2$.
It follows that, excluding finitely many $(y,y')$, the union in \eqref{eq:union} is four-dimensional. 

Consider one of the excluded points $(y_0,y_0')$.
If $F(x,y_0,z)$ were identically zero as a polynomial in $x,z$, then $F(x,y,z)$ would have a factor $y-y_0$, contradicting the irreducibility of $F$. Thus $F(x,y_0,z)$ is not identically zero, which implies that $V_{y_0y_0'}$ is at most two-dimensional.
Hence each of the finitely many excluded sets $\{(y_0,y_0')\}\times V_{y_0,y_0'}\times  V_{y_0,y_0'}$ is at most four-dimensional.
This finishes the proof.
\end{proof}

Our analysis also requires that the projection $\pi_1(V)$ of $V$ to the first four coordinates be 
four-dimensional. 
This fact, which does not follow directly from $V$ being four-dimensional, is established in Lemma~\ref{lem:piV=4}  below. 
The proof requires the following technical lemma. (The lemma and its proof should be compared to Lemma~\ref{lem:excepa}.)

\begin{lemma}\label{lem:excepb}
Let $H\in\C[u,v,w]$ be an irreducible polynomial of degree $d$ with $\partial H/\partial w$ not identically zero. Then
\[\dim\Big(\cl\big(\left\{(u_0,v_0)\in\C^2\mid \exists c~\text{such that}~H(u_0,v_0,w)\equiv c~\text{for all}~w\in \C \right\}\big)\Big)\le 1.\]
\end{lemma}
\begin{proof}
The case $d=1$ is easy, so we may assume that $d\ge 2$.
Write 
\[H(u,v,w)=\sum_{i=0}^d \alpha_i(u,v)w^i,\]
with suitable polynomials $\alpha_0(u,v),\ldots,\alpha_d(u,v)$. 
Observe that there exists some $i_0>0$ such that $\alpha_{i_0}$ is not identically zero,
for otherwise $\partial H/\partial w$ would be identically zero, contrary to assumption.
If $H(u_0,v_0,w)\equiv c$ for some $u_0,v_0,c$, then $\alpha_i(u_0,v_0) = 0$ for $i>0$, and in particular $(u_0,v_0)\in Z(\alpha_{i_0})$.
Hence the variety in the statement of the lemma is contained in $Z(\alpha_{i_0})$, which implies that it has dimension at most $1$.
\end{proof}

\begin{lemma}\label{lem:piV=4}
We have $\cl(\pi_1(V))=\C^4$.
\end{lemma}
\begin{proof}
Let $(x_0,x_0',y_0,y_0')\in\C^4$. There exist $z_1,z_2,z_3,z_4\in\C$ such that
\[
F(x_0,y_0,z_1)=F(x_0,y_0',z_2)=F(x_0',y_0,z_3)=F(x_0',y_0',z_4)=0,
\]
unless $F(x_0,y_0,z)\equiv c$ for some  nonzero $c\in \C$,
or a similar identity holds for one of the other pairs $(x_0,y_0')$, $(x_0',y_0)$, $(x_0',y_0')$.
In other words, we have $(x_0,x_0',y_0,y_0')\in\pi_1(V)$ unless one of these exceptions holds.

Let 
$$
\sigma:=\cl\left(\{(x_0,y_0)\in\C^2\mid \exists c~\text{such that}~ F(x_0,y_0,z)\equiv c\}\right)
$$
(note that here we include the case $c=0$). 
By Lemma~\ref{lem:excepb}
we have $\dim(\sigma)\leq 1$, so the set
\[
\sigma':=\left\{(x,x',y,y')\mid \text{one of}~(x,y), (x,y'),(x',y),(x',y')~\text{is in}~\sigma\right\}
\]
has dimension at most $3$.
It follows that $\cl(\C^4\minus\sigma')=\C^4$ (see Lemma~\ref{lem:dense}).
As observed above, we have $\C^4\minus \sigma'\subset \pi_1(V)$, so 
we conclude that $\cl(\pi_1(V))=\C^4$.
\end{proof}

We will use the implicit function theorem (spelled out in Lemma~\ref{lem:ift2})
to locally express each 
of the variables $z_1,z_2,z_3,z_4$ in terms of the corresponding pair of the first four variables $x,x',y,y'$. 
To facilitate this we first exclude the subvariety of $V$ defined by
$$V_0:=V_1\cup V_2\cup V_3,$$ where
\[V_i:=\left\{(x,x',y,y',z_1,z_2,z_3,z_4)\in V\mid F_i(x,y,z_1)F_i(x,y',z_2)F_i(x',y,z_3)F_i(x',y',z_4)=0\right\},
\] 
and $F_i$ stands for the derivative of $F$ with respect to its $i$th variable, for $i=1,2,3$.\footnote{Note 
that, just to apply the implicit function theorem, it would suffice to exclude $V_3$; 
we exclude $V_1$ and $V_2$ for technical reasons that arise in the proofs of Lemmas \ref{lem:Vgcont} and \ref{lem:special} below.}

Let $\rho:\C^8\to \C^6$ be the (permuted) projection 
$$\rho:(x,x',y,y',z_1,z_2,z_3,z_4)
\mapsto (x,y,z_1, x',y',z_4).$$
We now show that $\cl(\rho(V_0))$ is at most three-dimensional, from which it follows that $\cl (\pi_1(V_0))$ is also at most three-dimensional, 
which will allow us to exclude it in most of our analysis. 

\begin{lemma}\label{lem:vzero}
The variety $\cl(\rho(V_0))$ has dimension at most 3.
\end{lemma}
\begin{proof}
It suffices to show that, for 
$$
\overline{V}_i:=\{(x,x',y,y',z_1,z_2,z_3,z_4)\in V\mid F_i(x,y,z_1)=0\},
$$ 
$\cl(\rho(\overline{V}_i))$ is at most three-dimensional, for each $i=1,2,3$; the other nine cases can be treated symmetrically.
So let $i$ be fixed. 
Then $\overline{V}_i$ is the common zero set of the five polynomials 
$$
F(x,y,z_1),\;F(x,y',z_2),\;F(x',y,z_3),\;F(x',y',z_4),\; F_i(x,y,z_1).
$$
The polynomial $F(x,y,z_1)$ is assumed to be irreducible,
and $F_i(x,y,z_1)$ has lower degree and is not identically zero by assumption. 
It follows that $F(x,y,z_1)$ and $F_i(x,y,z_1)$ are coprime polynomials,
so the equations $F(x,y,z_1)=0,F_i(x,y,z_1)=0$ define a one-dimensional variety in $\C^3$ (with coordinates $x,y,z_1$). 
Clearly, 
$F(x',y',z_4)=0$ defines a two-dimensional variety in a complementary copy of $\C^3$ (with coordinates $x',y',z_4$). Thus 
\[
\widehat{V}_i:=\{(x,y,z_1)\mid F(x,y,z_1)=F_i(x,y,z_1)=0\}\times \{(x',y',z_4)\mid F(x',y',z_4)=0\},
\]
viewed as a variety in $\C^6=\C^3\times\C^3$, with coordinates $x,y,z_1,x',y',z_4$, is three-dimensional. 
Since $\rho(\overline{V}_i)\subset \widehat{V}_i$, and since the closure operation does not increase the dimension
(Lemma \ref{lem:projectionpreserves}), the lemma follows.
\end{proof}

As explained in the overview in Section~\ref{sec:overview}, we want to view $V$, around most of its points, 
as the graph of a locally defined mapping.  We now define this mapping.

\begin{lemma}\label{lem:locallygraph}
For each point $v\in V\minus V_0$,
there is an open neighborhood $N_v\subset \C^8$ of $v$, 
disjoint from $V_0$,
and an analytic mapping
$\Phi_v : \pi_1(N_v)\to \pi_2(N_v),$
such that 
$$
V\cap N_v = \{(u,\Phi_v(u))\mid u\in \pi_1(N_v)\}.
$$
\end{lemma}
\begin{proof}
Let $v=(a,a',b,b',c_1,c_2,c_3,c_4)\in V\minus V_0$ be an arbitrary point.
We apply the implicit function theorem in $\C^3$ (Lemma \ref{lem:ift2}) to the equation $F(x,y,z_1)=0$ at the point $(a,b,c_1)$.
Since $v\not\in V_0$, we have 
$F_3(a,b,c_1)\neq 0$.
We thus obtain neighborhoods $U_{a,b}$ of $(a,b)$ in $\C^2$ and $U_{c_1}$ of $c_1$ in $\C$, 
and an analytic mapping $g_1:U_{a,b}\to U_{c_1}$ such that 
\[
\{(x,y,z_1)\in U_{a,b}\times U_{c_1} \mid F(x,y,z_1) =0\}=\{(x,y,g_1(x,y)) \mid (x,y)\in U_{a,b}\}.
\]

We can do the same at each of the points $(a,b',c_2),(a',b,c_3),(a',b',c_4)$, leading to analogous mappings $g_2,g_3,g_4$.
It follows that we can find neighborhoods $N_1$ of $a$, $N_2$ of $a'$, $N_3$ of $b$, and $N_4$ of $b'$, and $N_1'$, $N_2'$, $N_3'$, $N_4'$ of $c_1$, $c_2$, $c_3$ $c_4$, respectively, such that 
\[
\Phi_v:(x,x',y,y')\mapsto\big( g_1(x,y),~g_2(x,y'),~g_3(x',y),~g_4(x',y') \big)
\]
defines an analytic map from $N_1\times N_2\times N_3\times N_4$ to $N_1'\times N_2'\times N_3'\times N_4'$.
Then 
\[N_v:=N_1\times N_2\times N_3\times N_4\times N_1'\times N_2'\times N_3'\times N_4', 
\] 
is a neighborhood of $v$ in $\C^8$ satisfying the conclusion of the lemma. If needed, we can shrink it to be disjoint from $V_0$.
\end{proof}

Let $G$ be the polynomial in $\C[x,x',y,y',z_1,z_2,z_3,z_4]$ given by 
\begin{align*}\label{eqq}
G = F_2(x,y,z_1)F_1(x,y',z_2)&F_1(x',y,z_3)F_2(x',y',z_4)\\
&-F_1(x,y,z_1)F_2(x,y',z_2)F_2(x',y,z_3)F_1(x',y',z_4).
\end{align*}
Consider the  subvariety $W :=V\cap Z(G)$ of $V$.
The significance of $W$ (and of $G$) lies in the following lemma; 
it says that $W$ is the subvariety of points $v\in V$ such that the determinant of the Jacobian of the local map $\Phi_v$ vanishes at $v$.
See Section \ref{sec:analysis} for the definition of the Jacobian matrix $J_{\bf f}$ of a map ${\bf f}$.

\begin{lemma}\label{lem:locallysingular}
For $v\in V\minus V_0$ we have $v\in W$ if and only if $\det(J_{\Phi_v}(\pi_1(v)) )=0$.
\end{lemma}
\begin{proof}
We write $g_{ij}$ for the derivative of the function $g_i$ (from the proof of Lemma \ref{lem:locallygraph}) 
with respect to its $j$th variable, for $i=1,2,3,4$ and $j=1,2$. The Jacobian matrix of $\Phi_v$, evaluated at 
$u=(x,x',y,y')\in \pi_1(N_v)$, where $N_v$ is the neighborhood of $v$ given in Lemma \ref{lem:locallygraph}, equals
\begin{equation}\label{eq:detzero}
J_{\Phi_v}(u)=\left(
\begin{array}{cccc}
g_{11}(x,y)	& g_{21}(x,y')	& 0				& 0\\
0			& 0				& g_{31}(x',y)	& g_{41}(x',y')\\
g_{12}(x,y)	& 0				& g_{32}(x',y)	& 0\\
0			& g_{22}(x,y')	& 0				& g_{42}(x',y')\\
\end{array}\right),
\end{equation}
or, by implicit differentiation,
\[
J_{\Phi_v}(u)=
\left(\begin{array}{cccc}
-\frac{F_1(x,y,z_1)}{F_3(x,y,z_1)}& -\frac{F_1(x,y',z_2)}{F_3(x,y',z_2)}&0&0\\
0&0&-\frac{F_1(x',y,z_3)}{F_3(x',y,z_3)}&-\frac{F_1(x',y',z_4)}{F_3(x',y',z_4)}\\
-\frac{F_2(x,y,z_1)}{F_3(x,y,z_1)}&0&-\frac{F_2(x',y,z_3)}{F_3(x',y,z_3)}&0\\
0&-\frac{F_2(x,y',z_2)}{F_3(x,y',z_2)}&0&-\frac{F_2(x',y',z_4)}{F_3(x',y',z_4)}
\end{array}\right),
\]
for $z_1=g_1(x,y)$, $z_2=g_2(x,y')$, $z_3=g_3(x',y)$, and $z_4=g_4(x',y')$.
Since $N_v\cap V_0=\emptyset$, all the denominators are non-zero (and, for that matter, so are all the numerators).
Write $v=(a,a',b,b',c_1,c_2,c_3,c_4)$ and observe that, by construction, 
$c_1=g_1(a,b)$, $c_2=g_2(a,b')$, $c_3=g_3(a',b)$, and $c_4=g_4(a',b')$.
Computing the determinant explicitly at the point $u=\pi_1(v)=(a,a',b,b')$ and clearing denominators gives exactly $G(v)$, where $G$ is the polynomial defining $W$.
Thus, 
$
\det J_{\Phi_{v}}(\pi_1(v))=0
$
if and only if 
$G(v)=0$.
\end{proof}


\subsection{The varieties \texorpdfstring{$V_\gamma$}{Vgamma}}

We now make precise what it means for a popular curve to be infinitely popular.

\begin{definition}
Let $\gamma\subset\C^2$ be a popular curve.
An irreducible curve $\gamma^*\subset \C^2$ is an \emph{associated curve} of $\gamma$ if for all but finitely many $(y,y')\in \gamma^*$ we have $\gamma\subset \gamma_{y,y'}$.
\end{definition}

The notation is meant to suggest that an associated curve $\gamma^*$ resembles a dual curve; however, like a popular curve $\gamma$, an associated curve may be a component of a dual curve, so strictly speaking the notions are separate.

Throughout this section, we let $\gamma$ be a popular curve and $\gamma^*$ an associated curve of $\gamma$.
In Section~\ref{sec:associatedcurves}, we will show that every popular curve has at least one associated curve.
As described in Section \ref{sec:overview}, we wish to show that if $\gamma$ is a popular curve and $\gamma^*$ is an associated curve, then locally $\Phi_v$
maps the three-dimensional set $\C^2\times \gamma^*$ into the two-dimensional set $\gamma\times\gamma$.
We will do this by showing that the intersection of $V$ and $\C^2\times \gamma^*\times\gamma\times\gamma$ is contained in $W$.

With each $\gamma\in\mathcal C$ and associated curve $\gamma^*$, we associate the variety 
\[
V_{\gamma}:= V\cap (\C^2\times \gamma^*\times\gamma\times\gamma)\subset \C^8.
\]
Note that $V_\gamma$ also depends on the choice of $\gamma^*$ (which is not necessarily unique), 
but we have suppressed this in the notation, sticking in what follows to some fixed $\gamma^*$.


\begin{lemma}\label{lem:Vgcont}
For all $\gamma\in\mathcal C$ we have $V_\gamma\subset W\cup V_0$ (for any choice of $\gamma^*$).
\end{lemma}
\begin{proof}
Let $\gamma_r^*$, $\gamma_r$ denote the subsets of regular points of $\gamma^*$, $\gamma$, respectively, and define
$$V_\gamma':= V\cap (\C^2\times \gamma^*_r\times\gamma_r\times\gamma_r).
$$
It is sufficient to show that $V_\gamma'\subset W\cup V_0$.
Indeed, since $W$ and $V_0$ are varieties, it follows that the closure of $V\cap (\C^2\times \gamma^*_r\times\gamma_r\times\gamma_r)$ is contained in $W\cup V_0$. 
This closure equals $V_\gamma$, since the removed points form a lower-dimensional subset.

Let $v\in V_\gamma'\minus V_0$ and assume, for contradiction, that $v\not\in W$.
Lemma \ref{lem:locallygraph} gives an open neighborhood $N_v$ of $v$, disjoint from $V_0$,
so that $V\cap N_v$ is the graph of an analytic map $\Phi_v:B_1\to B_2$, where $B_1 :=\pi_1(N_v)$ and $B_2:=\pi_2(N_v)$.
By Lemma \ref{lem:locallysingular}, $\det (J_{\Phi_v}(\pi_1(v)))\neq 0$.
By the inverse mapping theorem 
(Lemma \ref{lem:imt}),
$\Phi_v$ is bianalytic on a sufficiently small neighborhood of $\pi_1(v)$, which, by shrinking $N_v$ if needed, we may assume to be $B_1$.  
Consider the mapping $\overline{\Phi}_v:=\Phi_v\circ \pi_1$ restricted to $V\cap N_v$. 
Note that $\overline{\Phi}_v$ is bianalytic. 
Indeed, $\pi_1$ restricted to $V\cap N_v$ is clearly bianalytic (its inverse is $u\mapsto (u,\Phi_v(u))$),
so $\overline{\Phi}_v$ is the composition of two bianalytic functions, hence itself bianalytic.
By definition of $V_\gamma$ we have  
$\overline{\Phi}_v(V_\gamma\cap N_v)\subset \gamma\times \gamma.$
Write $v=(a,a',b,b',c_1,c_2,c_3,c_4)$, and note that,
by definition of $V_\gamma'$, $(c_1, c_2),(c_3, c_4)$ are regular points of $\gamma$ and $(b,b')$ is a regular point of $\gamma^*$.

We claim that there exists an open set $N\subset N_v$ such that $V_\gamma\cap N$ is locally three-dimensional.
Indeed, we may assume, without loss of generality, that none of the tangents to $\gamma$ at $(c_1,c_2)$, $(c_3,c_4)$, and to $\gamma^*$
at $(b,b')$ are vertical in the respective planes (otherwise, we simply switch the roles of the first and the second coordinate in the relevant copy of $\C^2$).
Applying the implicit function theorem (Lemma \ref{lem:ift1}) 
to $\gamma$ and $\gamma^*$ at these regular points, we may therefore write 
$z_2=\rho_1(z_1)$, $z_4=\rho_2(z_3)$, and $y'=\rho_3(y)$
in sufficiently small neighborhoods of $(c_1,c_2)$, $(c_3,c_4)$, $(b,b')$ along the respective curves, 
for suitable analytic functions $\rho_1,\rho_2,\rho_3$.
Similarly, applying the implicit function theorem to $Z(F)$ in sufficiently small neighborhoods of $(a,b,c_1)$, $(a',b,c_3)$ 
(which we may, since we are away from $V_0$), we may write
$x=\sigma_1(y,z_1)$, $x'=\sigma_2(y,z_3)$, for analytic functions $\sigma_1, \sigma_2$.
Combining the functions above, we obtain an open neighborhood $N$ of $v$ such that the map
\[
(y,z_1,z_3)\mapsto (\sigma_1(y,z_1), \sigma_2(y,z_3),y,\rho_3(y), z_1,\rho_1(z_1),z_3,\rho_2(z_3))
\]
is bianalytic from an open neighborhood of $(b,c_1,c_3)$ to $V_\gamma\cap N$.
This implies that $V_\gamma\cap N$ is locally three-dimensional.  Since $\gamma\times \gamma$ has local dimension $2$ at every 
pair of regular points, and $\overline{\Phi}_v$ preserves local dimension, since it is bianalytic, this yields a contradiction,
which completes the proof of the lemma.
\end{proof}

Next we show that the projection of $V_\gamma$ to the first four coordinates is three-dimensional.

\begin{lemma}\label{lem:projVg}
Let $\gamma\in\mathcal C$ and assume that $\gamma$ is not an axis-parallel line.
Then 
$$\cl(\pi_1(V_\gamma))=\C^2\times \gamma^*.
$$ 
\end{lemma}
\begin{proof}
We clearly have 
$$\pi_1(V_\gamma)\subseteq \pi_1(\C^2\times \gamma^*\times\gamma\times\gamma) = \C^2\times \gamma^*,$$
so, since $\C^2\times \gamma^*$ is a variety, we get 
$$\cl(\pi_1(V_\gamma))\subseteq \C^2\times \gamma^*.$$

By definition (and Lemma~\ref{lem:quantelim}),
there is a finite subset $S\subset \gamma^*$ such that, for all $(b,b')\in \gamma^*\minus S$, 
$\gamma_{b,b'}$ is a curve and $\gamma\subset \gamma_{b,b'}$. 
It follows from the definitions of $V_\gamma$ and $V$ that
\begin{equation}\label{eq:betas}
\pi_1(V_\gamma)\supseteq \bigcup_{(b,b')\in\gamma^*\minus S}
\beta_{b,b'}\times \beta_{b,b'} \times\{(b,b')\},
\end{equation}
where
\[
\beta_{b,b'}:=\{x\in \C\mid \exists (c_1,c_2)\in\gamma~\text{such that}~F(x,b,c_1)=F(x,b',c_2)=0\}.
\]

We claim that $\cl(\beta_{b,b'})  =\C$. For this, it is sufficient to show that $\beta_{b,b'}$ is infinite. 
Assume to the contrary that it is finite. In this case $\gamma$ is contained in the closure of the finite union
\[
\bigcup_{x\in \beta_{b,b'}}\big\{(z_1,z_2)\in \C^2 \mid F(x,b,z_1)=F(x,b',z_2)=0\big\}.
\]
Since $\gamma$ is infinite, one of the sets in the union must be infinite. Then $\gamma$ must be a line parallel to one of the axes in $\C^2$, contradicting the assumption of the lemma.

Hence 
\begin{align*}
\cl\Big(\bigcup_{(b,b')\in\gamma^*\minus S}
\beta_{b,b'}\times \beta_{b,b'} \times\{(b,b')\}
\Big) &\supseteq
\bigcup_{(b,b')\in\gamma^*\minus S}\cl\big(
\beta_{b,b'}\times \beta_{b,b'} \times \{(b,b')\}
\big)\\
&=\C^2\times \cl\Big(\bigcup_{(b,b')\in \gamma^*\minus S} \{(b,b')\}\Big)
\end{align*}
using that the closure of an infinite union \emph{contains} the union of the closures, 
and that the closure of a product is the product of the closures.
Together with \eqref{eq:betas} this gives
\[\cl(\pi_1(V_\gamma))\supset\C^2\times \cl\big(\gamma^*\minus S\big) =\C^2\times \gamma^*,\]
completing the proof of the lemma.
\end{proof}



\subsection{The associated curves}\label{sec:associatedcurves}

In this section we show that if a curve $\gamma$ is popular, then it must be {\em infinitely} popular.
First we need the following sharpened form of B\'ezout's inequality for many curves. 
Our proof is adapted from Tao \cite{Tao11b}.
See Appendix \ref{sec:degree} for the definition of degree.
Note that for a reducible variety, its degree is the sum of the degrees of its irreducible components; in particular, 
if a one-dimensional variety contains zero-dimensional components (isolated points), then the degree is the sum of the 
degrees of the purely one-dimensional irreducible components, plus the number of zero-dimensional components.

\begin{lemma}[{\bf B\'ezout for many curves}]
\label{lem:bezoutmany}
If $\mathcal F$ is a (possibly infinite) family of algebraic curves in $\C^2$, each of degree at most $\delta$, then
\[\deg\Big(\bigcap_{C\in\mathcal F}C\Big)\leq \delta^2.\]
In other words, either 
$\bigcap_{C\in\mathcal F}C$ is zero-dimensional and has cardinality at most $\delta^2$, 
or it has dimension 1 and degree at most $\delta^2$.
\end{lemma}
\begin{proof}
The variety $X:= \bigcap_{C\in\mathcal F}C$ has dimension either $0$ or $1$.
If $X$ is zero-dimensional, then it is a finite set of points.
By the Noetherian property (see Harris~\cite[p.\,18]{Ha92}), 
there is a finite subset of curves $C_1,\ldots,C_s\in \mathcal{F}$ such that $X = \bigcap_{i=1}^s C_i$. We must have $s\ge 2$.

Each irreducible component of $C_1$ either has finite intersection with $C_2$, or it is also a component of $C_2$.
Thus we can write $C_1 = D_2\cup E_2$,
 with curves (or empty sets) $D_2$ and $E_2$, such that $D_2\cap C_2$ is finite and $E_2$ is the union of some irreducible components of $C_2$.
For $i =3,\ldots,s$, we inductively partition $E_{i-1}$, which is a subset of $C_{i-1}$, as $D_i\cup E_i$, 
such that $D_i\cap C_i$ is finite and $E_i$ is the union of some irreducible components of $C_i$.
Then $E_s$ must be empty (or else $X$ would be one-dimensional), so $C_1 = \bigcup_{i=2}^s D_i$ and
$X \subseteq \bigcup_{i=2}^s \left(C_i\cap D_i \right)$,  
since each $x\in X$ is contained in $C_1$, so in some $D_i$, and also in $C_i$.
Therefore, $X$ is finite and, using the standard Bezout's inequality (Theorem \ref{thm:bezout}),
\[
|X| \leq \sum_{i=2}^s \deg(C_i)\cdot \deg(D_i) \leq \delta \cdot \sum_{i=2}^s \deg(D_i) 
\leq \delta^2.
\]
Here we have $\sum \deg(D_i)\leq \deg(C_1)$ because the $D_i$ are distinct subcurves of $C_1$.

Now suppose that $X$ is one-dimensional.
Let $X_1$ be the maximal purely one-dimensional subvariety, and let $X_0$ the remainder, which must be zero-dimensional.
Set $\delta_1:=\deg(X_1)$; since $X_1$ is a subset of any curve of $\mathcal{F}$, we have $\delta_1\leq \delta$.
Remove $X_1$ from every curve in $\mathcal{F}$, take the closure of each curve again, and let $\mathcal{F}'$ 
be the set of the resulting curves, each of which has degree at most $\delta-\delta_1$. By the choice
of $X_1$, we have that $\bigcap_{C\in \mathcal{F}'} C$ is zero-dimensional.
By applying the argument above to $\mathcal{F}'$, we get that $|X_0|\leq (\delta-\delta_1)^2$, and thus $\deg(X)\leq (\delta-\delta_1)^2+\delta_1 \leq \delta^2$.
\end{proof}

Recall that $\mathcal{C}$ is the set of popular curves, 
i.e., irreducible curves $\gamma$ that are contained in $\gamma_{y,y'}$ for more 
than $d^4$ points $(y,y')\in \C^2\minus\mathcal{S}$ (where $\mathcal S$ is the set constructed in Section~\ref{sec:proofofmain}).
Lemma~\ref{lem:associated} shows that if $\gamma$ is popular, then there is a one-dimensional set of curves $\gamma_{y,y'}$ that contain $\gamma$. 

\begin{lemma}\label{lem:associated}
Every $\gamma\in\mathcal{C}$ has at least one associated curve.
More precisely, for every $\gamma\in\mathcal C$ there exists an irreducible algebraic curve $\gamma^*\subset\C^2$ of degree at most $d^2$ such that for all but finitely many $(y,y')\in \gamma^*$ we have $\gamma\subset \gamma_{y,y'}$.
\end{lemma}
\begin{proof}
By definition of $\mathcal{C}$, if $\gamma\in \mathcal{C}$, then there exists a set $I\subset \C^2\minus\mathcal{S}$ of size $|I| = d^4+1$ such that $\gamma\subset \gamma_{y,y'}$ for all $(y,y')\in I$.
This means that for all $(y,y')\in I$ and for all but finitely many $(z,z')\in \gamma$, 
there is an $x\in \C$ such that $F(x,y,z)=F(x,y',z')=0$, which implies that $(y,y')\in \gamma^*_{z,z'}$.
Thus we have $I\subset \gamma^*_{z,z'}$ for all but finitely many $(z,z')\in \gamma$.

Let $\mathcal{F}$ be the infinite family of curves $\gamma^*_{z,z'}$ with $(z,z')\in \gamma$ and $I\subset \gamma^*_{z,z'}$, and define
\[S_I := \bigcap_{\gamma^*_{z,z'}\in \mathcal{F}} \gamma^*_{z,z'}. \]
Then we have $I\subset S_I$.
Since all the curves in $\mathcal{F}$ have degree at most $d^2$, Lemma \ref{lem:bezoutmany} implies that $S_I$ has degree at most $d^4$.
Since $|I|>d^4$, $S_I$ must have dimension 1.
Let $\gamma^*$ be any irreducible one-dimensional component of $S_I$.

If $(y,y')\in \gamma^*$, then for all but finitely many $(z,z')\in \gamma$ we have $(y,y')\in \gamma^*_{z,z'}$.
It follows that for all but finitely many $(y,y')\in \gamma^*$, and for all but finitely many $(z,z')\in \gamma$ 
(where the excluded points $(z,z')$ depend on the choice of $(y,y')$), 
we have $(z,z')\in \gamma_{y,y'}$.
Since both $\gamma$ and $\gamma_{y,y'}$ are algebraic curves, and $\gamma$ is irreducible, it follows that $\gamma\subset \gamma_{y,y'}$ for all but finitely many $(y,y')\in \gamma^*$. 
This means that $\gamma^*$ is an associated curve of $\gamma$.
\end{proof}

In the next two sections we separate our analysis into two cases, according to the dimension of $\cl(\pi_1(W))$. 


\subsection{Case 1: \texorpdfstring{$\dim \cl(\pi_1(W))\leq 3$}{dim at most 3} implies few popular curves}\label{sec:case1}
Throughout this subsection we assume that $\dim\cl(\pi_1(W))\le 3$. In this case we establish the existence of the set 
$\mathcal{Y}$ in Proposition~\ref{prop:keyprop}(a) (and similarly the set $\mathcal{Z}$ in Proposition~\ref{prop:keyprop}(b)).

As the statement of Lemma \ref{lem:projVg} 
suggests, popular curves that are axis-parallel lines require a different treatment, provided by the following simple lemma.

\begin{lemma}\label{lem:axisparallel}
There is a one-dimensional variety $\mathcal{Y}_1\subset\C^2$ with $\deg(\mathcal{Y}_1)=O(d^2)$, 
containing $\mathcal{S}$, 
such that for every $(y_1,y_2)\in \C^2\minus \mathcal{Y}_1$
the curve $\gamma_{y_1,y_2}$ contains no axis-parallel line.
\end{lemma}
\begin{proof}
As in the proof of Lemma \ref{lem:quantelim}, define
\[R := \{y_0\mid  \exists x_0 ~\text{such that}~F(x_0,y_0,z)\equiv 0
~\text{or}~
 \exists z_0 ~\text{such that}~F(x,y_0,z_0)\equiv 0\}.\]
Assume that $(y_1,y_2)\not\in \mathcal{S}$, 
and suppose that the curve $\gamma_{y_1,y_2}$ (where we equip the plane in which it is defined with coordinates $z$, $z'$)
contains a horizontal line $Z(z'-z_2)$.
Then for each $z=z_1$ there is an $x_0$ such that $F(x_0,y_1,z_1) = 0$ and $F(x_0,y_2,z_2) = 0$.
If $F(x,y_2,z_2)$ is identically zero as a polynomial in $x$, then $y_2\in R$.
If $F(x,y_2,z_2)$ is not identically zero, there are only finitely many $x_0$ such that $F(x_0,y_2,z_2)=0$.
Since we assumed that $(y_1,y_2)\not\in \mathcal{S}$, $\gamma_{y_1,y_2}$ is one-dimensional, hence infinite, 
and thus for at least one $x_0$ satisfying $F(x_0,y_2,z_2)=0$ there must be infinitely many $z_1$ such that $F(x_0,y_1,z_1)=0$. 
This implies that $F(x_0,y_1,z)\equiv 0$ (as a polynomial in $z$), and thus $y_1\in R$.
By symmetry, 
if $\gamma_{y_1,y_2}$ contains a vertical line,
then in this case too we have either $y_1\in R$ or $y_2\in R$.
Thus, if we set 
$$
\mathcal{Y}_1:=(R\times \C) \cup (\C\times R),
$$
then $\gamma_{y_1,y_2}$ does not contain a horizontal or vertical line when $(y_1,y_2)\not\in \mathcal{Y}_1$.
By construction, $\deg(\mathcal{Y}_1) = O(d^2)$, because $\deg(R)=O(d^2)$, and we also have $\mathcal{S} = R\times R\subset \mathcal{Y}_1$.
\end{proof}

We also need the following observation.

\begin{lemma}\label{lem:associatedmult}
An irreducible curve $\gamma^*$ is associated to at most $d^2$ curves $\gamma \in \mathcal{C}$.
\end{lemma}
\begin{proof}
Suppose there is a set $\mathcal{C}'$ of $d^2+1$ distinct curves $\gamma \in \mathcal{C}$ that $\gamma^*$ is associated to.
For each $\gamma\in \mathcal{C}'$, we have that, for all but finitely many $(y,y')\in \gamma^*$, $\gamma$ is contained in $\gamma_{y,y'}$.
It follows that there is a point $(y,y')\in \gamma^*$ such that $\gamma\subset \gamma_{y,y'}$ for all $\gamma\in \mathcal{C}'$.
This is a contradiction, because $\gamma_{y,y'}$ has at most $d^2$ irreducible components by Lemma \ref{lem:degreebound}.
\end{proof}

We are now ready to prove the key fact that the number of popular curves is bounded.

\begin{lemma}\label{lem:fewpopular}
There are at most $O(d^7)$ distinct popular curves $\gamma\in\mathcal{C}$ that are not axis-parallel lines.
\end{lemma}
\begin{proof}
Let $\gamma\in\mathcal C$, 
assume that it is not an axis-parallel line, 
and let $\gamma^*$ be an associated curve of $\gamma$.
Since  $\gamma^*$ is irreducible, $\C^2\times\gamma^*$ is an irreducible variety. 
Using Lemma \ref{lem:projVg} and Lemma \ref{lem:Vgcont}, we have
\[
\C^2\times \gamma^*=\cl (\pi_1(V_{\gamma}))\subset \cl(\pi_1(W\cup V_0)) =X\cup Y,
\]
for $X:=\cl(\pi_1(W))$ and $Y:=\cl(\pi_1(V_0))$. 
We have $\dim(X)\leq 3$ by the assumption in this subsection, 
and $\dim(Y)\leq 3$ by Lemma \ref{lem:vzero}.
We also have $\deg(X) =O(d^5)$ and $\deg(Y) = O(d^5)$ by Lemma \ref{lem:degreebound} and Lemma \ref{lem:projectionpreserves} from the appendix, 
since both are closures of projections of varieties defined by five polynomials, each of degree at most\footnote{In fact,
all of them are of degree $\le d$, except for $G$, which is of degree $\le 4(d-1)$.}
 $O(d)$.
Since $X\cup Y$ is at most three-dimensional, and each $\C^2\times\gamma^*$ is an irreducible 
three-dimensional subvariety of $X\cup Y$, it follows that $\C^2\times \gamma^*$ is one of the finitely many irreducible components of $X\cup Y$.

Let $T$ be the set of all associated curves of all curves $\gamma\in\mathcal{C}$ (excluding those $\gamma$ that are axis-parallel lines).
The preceding argument shows that $T$ is a finite set.
Moreover, we have
\[\sum_{\gamma^*\in T} \deg(\gamma^*)
= \sum_{\gamma^*\in T} \deg(\C^2\times \gamma^*)
\leq \deg(X\cup Y)= O(d^5).\]
This implies that the total number of distinct associated curves is $O(d^5)$.
Since by Lemma \ref{lem:associated} each popular curve has at least one associated curve, 
and by Lemma \ref{lem:associatedmult} each associated curve is associated to at most $d^2$ popular curves,
it follows that the number of popular curves is bounded by $O(d^7)$.
\end{proof}

Finally, we show that the union of all the associated curves (which are not axis-parallel lines) has bounded degree.

\begin{lemma}\label{lem:X2}
Let 
\[\mathcal{Y}_2:=\cl\Bigl(\Bigl\{(y,y')\in\C^2\mid
\exists \gamma\in \mathcal{C},~\text{not an axis-parallel line, such that}~\gamma\subset\gamma_{y,y'}\Bigr\}\Bigr).\]
Then $\mathcal{Y}_2$ is one-dimensional; its purely one-dimensional portion has degree $O(d^7)$, 
and the number of zero-dimensional components is $O(d^{11})$.
\end{lemma}
\begin{proof}
By construction, any one-dimensional irreducible component of $\mathcal{Y}_2$ is an associated curve.
In the proof of Lemma \ref{lem:fewpopular} it was shown that there are at most $O(d^5)$ associated curves.
Since each associated curve has degree $O(d^2)$, the purely one-dimensional portion of $\mathcal{Y}_2$ has degree $O(d^7)$.

We next bound the number of zero-dimensional components of $\mathcal{Y}_2$, which we refer to as \emph{associated points};
such a point cannot have been added by the closure, so it must be associated to some popular curve.
By Lemma \ref{lem:fewpopular}, the number of popular curves $\gamma\in\mathcal{C}$ is $O(d^7)$. 
We show that each of these has at most $d^4$ associated points.
Let $\gamma\in \mathcal{C}$ and suppose that $\gamma$ has $d^4+1$ associated points that do not lie on associated curves.
That is, these points form a set $I\subset \C^2\minus\mathcal{S}$ of size $|I| = d^4+1$, such that $\gamma\subset \gamma_{y,y'}$ for all $(y,y')\in I$.
Exactly as in the proof of Lemma \ref{lem:associated}, there is a curve $S_I$, which is the intersection of an infinite family of curves $\gamma^*_{z,z'}$ containing $I$.
Thus we have $I\subset S_I$.
As shown in that proof, each one-dimensional irreducible component of $S_I$ is an associated curve, so does not cover any point of $I$.
By Lemma \ref{lem:bezoutmany}, $S_I$ has degree at most $d^4$, and therefore contains at most $d^4$ isolated points.
This contradicts the fact that $|I|>d^4$.
\end{proof}

We put 
$\displaystyle{\mathcal{Y}:=\mathcal{Y}_1\cup\mathcal{Y}_2}$.
Combining Lemma \ref{lem:axisparallel} and Lemma \ref{lem:X2}, we get $\dim(\mathcal{Y}) = 1$ and $\deg(\mathcal{Y})=O(d^{11})$.
From the definitions of $\mathcal{Y}_1$ and $\mathcal{Y}_2$ it is clear that for $(y,y')\not\in \mathcal{Y}$, the curve $\gamma_{y,y'}$ does not contain any popular curve.
This completes the proof of Proposition \ref{prop:keyprop}(a) in Case 1.
Proposition \ref{prop:keyprop}(b) is proved in a fully symmetric manner.

\subsection{Case 2: \texorpdfstring{$\dim \cl(\pi_1(W))=4$}{dim is four} implies a special form of \texorpdfstring{$F$}{F}}
\label{sec:case2}
Throughout this subsection we assume that $\dim \cl(\pi_1(W))=4$.
Since $W$ is a subvariety of $V$, which is four-dimensional by Lemma \ref{lem:Vdim4}, $W$ must also be four-dimensional.
This implies that there exists an irreducible component $V'$ of $W$ (and of $V$) such that $\dim V'= 4$ and $\cl(\pi_1(V'))=\C^4$.
We will work only with $V'$ in the rest of this subsection.

We first show that most points of $Z(F)$, excluding only a lower-dimensional subset,
can be extended to points of $V'$, in the following sense.

\begin{lemma}\label{lem:extend}
There exists a one-dimensional subvariety $Z_0\subset Z(F)$ such that, for every $(a,b,c_1)\in Z(F)\minus Z_0$, there exist $a',b',c_2,c_3,c_4$
 such that $(a,a',b,b',c_1,c_2,c_3,c_4)$ is a regular point of $V'$ which is not in $V_0$.
\end{lemma}
\begin{proof}
Recall the definition of the (permuted) projection $\rho:\C^8\to \C^6$  
$$\rho:(x,x',y,y',z_1,z_2,z_3,z_4)
\mapsto (x,y,z_1, x',y',z_4)$$
We claim that 
$\cl(\rho(V'))=Z(F)\times Z(F)$.
Since $Z(F)\times Z(F)$ is four-dimensional and irreducible, and since, by definition of $V$, $\rho(V') \subset Z(F)\times Z(F)$, 
it suffices to prove that $\cl(\rho(V'))$ is (at least) four-dimensional.
We observe that 
$\sigma(\rho(V'))=\pi_1(V')$, 
where 
$$\sigma:(x,y,z_1, x',y',z_4)\mapsto (x, x',y,y').$$
Since the projection $\sigma$ does not increase the dimension (see Lemma \ref{lem:projectionpreserves}), we have
$$\dim \cl(\rho(V'))\ge \dim \cl(\pi_1(V'))=4,$$ 
proving our claim.

Using Lemma~\ref{lem:projectionclosure}, it follows that
$$
U_1:=\cl\big((Z(F)\times Z(F))\minus\rho(V')\big)=\cl\big(\cl(\rho(V'))\minus\rho(V')\big)
$$
is at most three-dimensional, and, in view of Lemma~\ref{lem:vzero}, $U_2:= \cl(\rho(V_0\cap V'))$ is also at most three-dimensional. 
Since $V'$ has dimension $4$,
the subvariety $V_s'$ of singular points of $V'$ is at most three-dimensional, so 
$U_3:=\cl(\rho(V_s'))$ is also at most three-dimensional. Hence, 
$$
U:=U_1\cup U_2 \cup U_3
$$ is a variety in $\C^6$ of dimension at most 3.
We set 
\begin{align*}
Z_0':
&=\left\{p\in Z(F)\mid
\dim\left(\left(\{p\}\times Z(F)\right)\cap U\right)\ge 2\right\}.
\end{align*}
In other words (using the fact that $\{p\}\times Z(F)$ is irreducible and two-dimensional), $p\in Z_0'$ if and only if $\{p\}\times Z(F)\subset U$,
so $Z_0'\times Z(F)\subset U$.
Since $U$ is a variety, 
we have 
$$\cl(Z_0')\times Z(F)=\cl(Z_0'\times Z(F))\subset U.$$
Set $Z_0:=\cl(Z_0')$.
Since $U$ is at most three-dimensional and $Z(F)$ is two-dimensional, we must have $\dim Z_0\le 1$.

Finally, let $(a,b,c_1)\in Z(F)\minus Z_0$. 
By definition of $Z_0$, we have  
$$\dim\left(\left(\{(a,b,c_1)\}\times Z(F)\right)\cap U\right)\le 1.$$
Thus there exists a point $(a',b',c_4)\in Z(F)$ such that $(a,b,c_1,a',b',c_4)\in (Z(F)\times Z(F))\minus U$.
By definition of $U$, this implies that $(a,b,c_1,a',b',c_4)\in\rho(V')\minus U$,
which in turn means that there exist
$c_2,c_3\in \C$ such that $(a,a',b,b',c_1,c_2,c_3,c_4)\in V'\minus V_0$, and  is a regular point of $V'$, as asserted.
\end{proof}

Let $Z_0$ be the variety given by Lemma~\ref{lem:extend}. 

\begin{lemma}\label{lem:special}
Let $u=(a,b,c_1)\in Z(F)\minus Z_0$.
Then there exist open sets $D_i\subset\C$ and analytic functions $\varphi_i:D_i\to\C$ with analytic inverses, 
for $i=1,2,3$, such that $(a,b,c_1)\in D_1\times D_2\times D_3$ and
\[
(x,y,z)\in Z(F)\quad\quad\text{if and only if}\quad\quad\varphi_1(x)+\varphi_2(y)+\varphi_3(z)=0,
\] 
for every $(x,y,z)\in D_1\times D_2\times D_3$.
\end{lemma}
\begin{proof}
By applying Lemma \ref{lem:extend} to $u=(a,b,c_1)$, we obtain $a',b',c_2,c_3,c_4\in\C$, such that the point 
$v:=(a,a',b,b',c_1,c_2,c_3,c_4)$ belongs to $V'\minus V_0$ and is regular in $V'$. 
By Lemma \ref{lem:locallygraph}, 
there exist neighborhoods $D$ of $a$, $D'$ of $a'$, $E$ of $b$, and $E'$ of $b'$, 
and a mapping
\[
\Phi_v:(x,x',y,y')\mapsto\big( g_1(x,y),~g_2(x,y'),~g_3(x',y),~g_4(x',y') \big),
\]
analytic on $D\times D'\times E\times E'$, such that its graph is the intersection $V\cap N_v$ for some open neighborhood $N_v$ of $v$ in $\C^8$.
Shrinking the sets $D$, $D'$, $E$, $E'$ as needed, the
image of $N_v$ under the projection $\pi_1$ can be assumed to be $D\times D'\times E\times E'$. 
Note that, since $v$ is a regular point of $V'$, $V'\cap N_v$ is necessarily four-dimensional, 
so it must coincide with $V\cap N_v$, if we take $N_v$ to be sufficiently small.
Thus, restricting the analysis  to the neighborhood $N_v$, we may use $V$ and $V'$ interchangeably in what follows.

Since $V'\subset W$, we have (recall that $W=Z(G)\cap V$)
$$
G(x,x',y,y',z_1,z_2,z_3,z_4)=0,
$$
for every $(x,x',y,y',z_1,z_2,z_3,z_4)\in V'\cap N_v$.
By the implicit function theorem (Lemma \ref{lem:ift2}), the functions $g_1,\ldots,g_4$ satisfy, in a suitable neighborhood of $v$, 
$$
g_{ij}(x,y)=-\frac{F_j(x,y,g_i(x,y))}{F_3(x,y,g_i(x,y))},$$
for $i=1,\ldots,4$ and $j=1,2$. By the definition of $G$,
this is easily seen to imply that
\[
g_{11}(x,y)g_{22}(x,y')g_{32}(x',y)g_{41}(x',y')
= 
g_{12}(x,y)g_{21}(x,y')g_{31}(x',y)g_{42}(x',y'),
\]
for every $(x,x',y,y')\in D\times D'\times E\times E'$.
In particular, fixing $x'=a'$ and $y'=b'$, there exists an open neighborhood $D\times E$ of $(a, b)\in\C^2$, such that
\begin{equation}\label{eq:deteq}
g_{11}(x,y)g_{22}(x,b')g_{32}(a',y)g_{41}(a',b')=g_{12}(x,y)g_{21}(x,b')g_{31}(a',y)g_{42}(a',b'),
\end{equation}
for every $(x,y)\in D\times E$.

Because $v\not\in V_0$, we have 
$$g_{11}(a,b)=-\frac{F_1(a,b,c_1)}{F_3(a,b,c_1)}\neq 0.
$$ 
Similarly,
$g_{22}(a,b')$, $g_{32}(a',b)$, $g_{41}(a',b')$, $g_{12}(a,b)$, $g_{21}(a,b')$, $g_{31}(a',b)$, and $g_{42}(a',b')$ are all nonzero. 
The continuity of all the relevant functions implies that, by shrinking $D\times E$ if needed, we may assume that, for all $(x,y)\in D\times E$, neither side of \eqref{eq:deteq} is zero.
Thus we can rewrite \eqref{eq:deteq} as
\begin{equation}\label{ratio1}
\frac{g_{11}(x,y)}{p(x)}=\frac{g_{12}(x,y)}{q(y)},
\end{equation}
where 
$$
p(x)=g_{21}(x,b')g_{42}(a',b')/g_{22}(x,b')\quad\text{and}\quad q(y)=g_{32}(a',y)g_{41}(a',b')/g_{31}(a',y)
$$
are analytic and nonzero on $D$ and $E$, respectively.
By Lang \cite[Theorem III.6.1]{Lan},
there exist analytic primitives $\varphi_1,\varphi_2$ so that $\varphi_1'(x)=p(x)$ on $D$ and $\varphi_2'(y)=q(y)$ on $E$.
Since, by construction, $\varphi_1',\varphi_2'$ are nonzero, and using the inverse mapping theorem, 
each of $\varphi_1,\varphi_2$ has an analytic inverse on its domain, possibly after shrinking $D$ and $E$ further.

We express the function $g_1(x,y)$ in terms of new coordinates $(\xi,\eta)$, given by
\begin{equation}\label{sys}
\xi=\varphi_1(x)+\varphi_2(y),\quad \eta=\varphi_1(x)-\varphi_2(y).
\end{equation}
Since $\varphi_1$, $\varphi_2$ 
are injections in suitable respective neighborhoods of $a$, $b$, we may assume that the system (\ref{sys}) is invertible in $D\times E$.
Returning to the standard notation, denoting partial derivatives by variable subscripts, we have
$$
\xi_{x}=\varphi_1'(x), ~~~
\xi_{y}
=\varphi_2'(y), ~~~
\eta_{x}
=\varphi_1'(x),~~~\text{and}~~~
\eta_{y}
=-\varphi_2'(y).
$$
Using the chain rule, we obtain
\[
g_{11}=g_{1\xi} \xi_x+g_{1\eta} \eta_x=\varphi_1'(x)(g_{1\xi}+g_{1\eta})=p(x)(g_{1\xi}+g_{1\eta})
\]
\[
g_{12}=g_{1\xi}\xi_y+g_{1\eta}\eta_y=\varphi_2'(y)(g_{1\xi}-g_{1\eta})=q(y)(g_{1\xi}-g_{1\eta}),
\]
which gives
$$\frac{g_{11}(x,y)}{p(x)}-\frac{g_{12}(x,y)}{q(y)}\equiv2g_{1\eta}(x,y),
$$
on $D\times E$. Combining this with (\ref{ratio1}), we get
$$
g_{1\eta}(x,y)\equiv 0.
$$
This means that $g_1$ depends only on the variable $\xi$, so it has the form
$$
g_1(x,y)=\psi(\varphi_1(x)+\varphi_2(y)),
$$
for a suitable analytic function $\psi$. 
The analyticity of $\psi$ is an easy consequence of the analyticity of $\varphi_1, \varphi_2$, and $g_1$, 
and the fact that $\varphi_1'(x)$ and $\varphi_2'(y)$ are nonzero,
combined with repeated applications of the chain rule.
Let 
$$S:=\{\varphi_1(x)+\varphi_2(y)\mid (x,y)\in D\times E\}~\quad\text{and}\quad T:=\{\psi(z)\mid z\in S\}.$$
We observe that 
$$
g_{11}(x,y)=\psi'(\varphi_1(x)+\varphi_2(y))\cdot p(x).
$$
As argued above, we have $g_{11}(x,y)\neq 0$ for all $(x,y)\in D\times E$, implying that $\psi'(\varphi_1(x)+\varphi_2(y))$ is nonzero for $(x,y)\in D\times E$.
Therefore, by the inverse mapping theorem (Lemma \ref{lem:imt}), $\psi:S\to T$ is invertible, possibly after shrinking $S$ and $T$.

Letting $\varphi_3(z):=-\psi^{-1}(z)$, we get, for $(x,y,z)\in D\times E\times T$, that
$$
\varphi_1(x)+\varphi_2(y)+\varphi_3(z)=0
$$
if and only if $(x,y,z)\in Z(F)\cap (D\times E\times T)$. 
This completes the proof of the lemma.
\end{proof}
Finally, Lemma \ref{lem:special} has established that $F$ satisfies property $(ii)$ of Theorem \ref{thm:main1}, which completes the proof of Proposition \ref{prop:keyprop}.
\qed


\section{An incidence bound}
\label{sec:incgeom}
For any finite set $\Pi\subset \C^2$ of distinct points and a finite set $\Gamma$ of curves in $\C^2$, define 
$$
I(\Pi,\Gamma) := |\{(p,\gamma)\in \Pi\times\Gamma\mid p\in \gamma\}|$$ 
to be the number of \emph{incidences} between the points and the curves.
If $\Gamma$ is a multiset, we count each incidence $(p,\gamma)$ with the multiplicity of $\gamma$.
A key tool in our proof is the following incidence bound, recently proved by Solymosi and De Zeeuw \cite{SZ15}.

\begin{theorem}
\label{thm:incprod}
Let $A_1,A_2\subset \C$ be finite and let $\Pi'\subseteq\Pi= A_1\times A_2$. Let $\Gamma$ be a finite set of algebraic curves of degree at most $\delta$ in $\C^2$,
without common components,
such that any pair of points of $\Pi'$ are both contained in at most $\mu$ curves of $\Gamma$.
Then
\[I(\Pi', \Gamma) = O\left(\delta^{4/3}\mu^{1/3}|\Pi|^{2/3}|\Gamma|^{2/3}+\mu|\Pi|+\delta^4|\Gamma|\right).\]
\end{theorem}

Theorem \ref{thm:incprod} is a special instance of the Szemer\'edi--Trotter bound 
(originally proved for points and lines in the real plane in \cite{ST83}) for complex ``pseudo-lines with bounded multiplicity'', 
under the strong assumption that the point set is (a subset of) a Cartesian product.
This assumption leads to a relatively simple proof, which also allows the dependence of $I(\Pi',\Gamma)$ on the parameters $\delta,\mu$ to be determined in the fairly sharp explicit form that we stated above.
For an arbitrary point set, the bound has been proved for complex lines by T\'oth \cite{Tot} and by Zahl \cite{Zah},
while for complex curves it 
was obtained by Solymosi and Tao \cite{SoT} and by Zahl \cite{Zah},
but only with an extra factor of the form $|\Pi|^\varepsilon$, under additional and fairly strong conditions, 
and without explicit dependence on the parameters $\delta$, $\mu$.
Recently, Sheffer and Zahl \cite{ShZa15} removed these conditions, but not the extra factor $|\Pi|^\varepsilon$.

The reason for stating Theorem \ref{thm:incprod} in terms of a subset $\Pi'\subseteq \Pi$ is that this is the type of set we get in the proof of Theorem~\ref{thm:main2}. 
One cannot simply replace $\Pi'$ by $\Pi$, because the condition concerning the number of curves 
passing through a pair of points depends on the specific choice of $\Pi'$. 
Fortunately, the result in \cite{SZ15} holds equally well for subsets (see \cite[Remark 15]{SZ15}), 
albeit resulting in an upper bound in terms of the size of the entire Cartesian product $\Pi$, as stated above.

Theorem~\ref{cor:multcor} is a generalization of Theorem~\ref{thm:incprod} to systems of points 
and curves that do not quite meet the conditions of the theorem, 
but that do satisfy them after appropriate decompositions of the sets of points and curves.
We repeat Definition~\ref{def:boundedmult} here for the convenience of the reader.

\begin{definition}
Let $\Pi$ be a finite set of distinct points in $\C^2$, 
and let $\Gamma$ be a finite multiset of curves in $\C^2$.
We say that the system $(\Pi,\Gamma)$ has \emph{$(\lambda,\mu)$-bounded multiplicity} 
if\\ 
$(a)$ for any curve $\gamma\in \Gamma$,
there are at most $\lambda$ curves $\gamma'\in \Gamma$ such that 
there are more than $\mu$ points contained in both $\gamma$ and $\gamma'$; and\\
$(b)$ for any point $p\in \Pi$, 
there are at most $\lambda$ points $p'\in \Pi$ such that there are more than $\mu$ curves that contain both $p$ and $p'$.\\
(In both $(a)$ and $(b)$ the curves should be counted with  multiplicity.)
\end{definition}


We now state and prove the incidence bound that we use in the proof of our main theorem.

\begin{theorem}\label{cor:multcor}
Let $A_1,A_2$ be finite subsets of $\C$ and $\Pi'\subseteq \Pi = A_1\times A_2$,
and let $\Gamma$ be a finite multiset of algebraic curves of degree at most $\delta$, such that the system $(\Pi,\Gamma)$ has $(\lambda,\mu)$-bounded multiplicity.
Then 
\[I(\Pi, \Gamma) 
=O\left(\delta^{4/3}\lambda^{4/3}\mu^{1/3}|\Pi|^{2/3}|\Gamma|^{2/3}+\lambda^2\mu|\Pi|+\delta^4\lambda|\Gamma|\right).
\]
\end{theorem}
\begin{proof}
We partition $\Pi'$ into subsets $\Pi_\alpha$ and $\Gamma$ into subsets $\Gamma_\beta$, so 
that for each pair $\alpha, \beta$, we can apply Theorem \ref{thm:incprod} to the system $(\Pi_\alpha,\Gamma_\beta)$.
Each of the subsets $\Gamma_\beta$ will consist of distinct elements, so we bypass the issue that $\Gamma$ might be a multiset.

Construct a graph $G$ with the curves in $\Gamma$ as vertices (where a curve that occurs $k$ times in the multiset corresponds to $k$ distinct vertices), 
and with an edge between any pair of curves that have intersection size larger than $\mu$
(thus all vertices representing the same curve form a clique).
By definition of $(\lambda,\mu)$-bounded multiplicity, $G$ has maximum vertex degree at most $\lambda$, 
so we can color the graph with $\lambda+1$ colors. In other words, we can partition the curves into $\lambda+1$ sets $\Gamma_\beta$, 
so that any pair of curves within the same $\Gamma_\beta$ have at most $\mu$ points of $\Pi'$ in common.
In particular, any pair of coinciding curves, or of curves that share some irreducible component, are placed in different color classes. 

We do the same for the points.
We construct a graph with the points of $\Pi'$ as vertices, and with an edge between any pair of points for which there are more than $\mu$ curves from $\Gamma$ containing both points.
Again the definition of $(\lambda,\mu)$-bounded multiplicity lets us color this graph with $\lambda+1$ colors.
We get subsets $\Pi_\alpha$ of $\Pi'$, such that for any two points in the same $\Pi_\alpha$, there are at most $\mu$ curves $\gamma\in\Gamma$ that pass through both of them.
Thus we can apply Theorem \ref{thm:incprod} for every pair $\alpha,\beta$
(making use of the fact that $\Pi_\alpha\subseteq \Pi'\subseteq\Pi$ and $\Pi$ is a Cartesian product)
 to obtain
\[I(\Pi_\alpha, \Gamma_\beta) = 
O\left(\delta^{4/3}\mu^{1/3}|\Pi|^{2/3}|\Gamma_\beta|^{2/3}+\mu|\Pi|+\delta^4|\Gamma_\beta|\right).
\]
Summing over all $(\lambda+1)^2$ pairs $\alpha,\beta$ and using H\"older's inequality 
gives the asserted bound.
\end{proof}

\section{The real case}\label{sec:real}


In this section we prove Theorem \ref{thm:main3}.
To avoid confusion, given a real polynomial $F$, in this section we write $Z_\R(F)$ for its \emph{real} zero set, and $Z_\C(F)$ for its \emph{complex} zero set.
We refer to Basu et al.~\cite{BPR03} for the relevant definitions.
Any regular point of a real zero set has a neighborhood where the zero set is a real manifold, and the dimension of the zero set is the maximum over all regular points of the dimension of that manifold. 
We recall that a real function $f:I\to \R$ on an interval $I$ is \emph{real-analytic} if it has a power series expansion at each point of $I$.

We want to prove that if $F\in \R[x,y,z]$ is irreducible over $\R$ and does not satisfy the 
unbalanced property $(i^*)$ of Theorem \ref{thm:main2} (over $\R$, and thus also over $\C$), then it satisfies property $(ii)_\R$ of Theorem \ref{thm:main3}.
We first treat the case where $F$ is irreducible over $\C$, and then the case where $F$ is reducible over $\C$, which requires some extra care.
The two respective lemmas that establish these properties give us Theorem \ref{thm:main3}.

\begin{lemma}\label{lem:realirred}
Let $F\in \R[x,y,z]$ be irreducible over $\C$ and assume that $Z_\R(F)$ has dimension 2.
If $F$ does not satisfy $(i^*)$, then it satisfies $(ii)_\R$.
\end{lemma}
\begin{proof}
By Theorem \ref{thm:main2}, property $(ii)$ holds (over $\C$), 
with an excluded one-dimensional subvariety $Y\subset Z_\C(F)$.
Let $Z_0$ be the union of $Y\cap\R^3$ with the set of singular points of $Z_\R(F)$; it is a variety of dimension at most 1.


Since property $(ii)$ holds over $\C$, for each $(a,b,c)\in Z_\R(F)\minus Z_{0}$,
there exist open sets $D_i\subset \C$ and analytic injections $\varphi_i:D_i\to \C$, 
such that $(a,b,c)\in D_1\times D_2\times D_3$, and, for every 
$(x,y,z)\in D_1\times D_2\times D_3$, we have
\begin{equation}\label{oritt}
(x,y,z)\in Z_\C(F) \quad\text{if and only}\quad \varphi_1(x)+\varphi_2(y)+\varphi_3(z)=0.
\end{equation}
By shrinking $D_1$, $D_2$, $D_3$, if needed, we may assume that
the sets $J_i:=D_i\cap \R$ are nonempty open intervals, containing, respectively, the points $a,b,c$.
Then we have $\left(D_1\times D_2\times D_3\right)\cap\R^3 = J_1\times J_2\times J_3$. 
Write $\varphi_j(x)=\mathsf{Re}(\varphi_j(x)) + i\cdot \mathsf{Im}(\varphi_j(x))$,
for $j=1,2,3$.
Each of the six functions
$\mathsf{Re}(\varphi_j)$, $\mathsf{Im}(\varphi_j)$, for $j=1,2,3$, is a real-analytic function,
since the real (or imaginary) part of a power series expansion for $\varphi_j$ is a real power series expansion for $\mathsf{Re}(\varphi_j)$ (or $\mathsf{Im}(\varphi_j)$).
By \eqref{oritt}, for every real point 
$(x,y,z)\in J_1\times J_2\times J_3$, we have $(x,y,z)\in Z_\R(F)$ if and only if
$$
\mathsf{Re}(\varphi_1(x)) + \mathsf{Re}(\varphi_2(y)) + \mathsf{Re}(\varphi_3(z)) =
\mathsf{Im}(\varphi_1(x)) + \mathsf{Im}(\varphi_2(y)) + \mathsf{Im}(\varphi_3(z)) = 0.
$$
Thus at least one of the two sets
\begin{align*}
\sigma_R & := \left\{ (x,y,z)\in J_1\times J_2\times J_3 \mid \mathsf{Re}(\varphi_1(x)) + \mathsf{Re}(\varphi_2(y)) + \mathsf{Re}(\varphi_3(z)) = 0 \right\}, \\
\sigma_I & := \left\{ (x,y,z)\in J_1\times J_2\times J_3 \mid \mathsf{Im}(\varphi_1(x)) + \mathsf{Im}(\varphi_2(y)) + \mathsf{Im}(\varphi_3(z)) = 0 \right\}
\end{align*}
does not equal $J_1\times J_2\times J_3$; say it is $\sigma_R$.
By what has just been shown, $\sigma_R$ contains
$Z_\R(F) \cap \left(J_1\times J_2\times J_3\right)$.

Since the point $(a,b,c)$ is not in $Z_0$, it is regular, so we can assume without loss of generality that $\partial F/\partial z(a,b,c)\neq 0$. 
It follows, by the real version of the implicit function theorem, that there is an open subset of $J_1\times J_2\times J_3$ in which $Z_\R(F)$ is the graph of a real-analytic function in $x,y$.
Since $\sigma_R$ is also the graph of a real-analytic function, it follows that $\sigma_R$ and $Z_\R(F)$ coincide in this open subset.
By shrinking this open subset we can assume that it is of the form $I_1\times I_2\times I_3$ for intervals $I_1\subset J_1,I_2\subset J_2,I_3\subset J_3$.
In other words, we have shown that, for every real point 
$(x,y,z)\in I_1\times I_2\times I_3$, 
$$
(x,y,z)\in Z_\R(F) \quad\text{if and only if}\quad
\mathsf{Re}(\varphi_1(x)) + \mathsf{Re}(\varphi_2(y)) + \mathsf{Re}(\varphi_3(z)) = 0 ,
$$
thereby showing that property $(ii)_\R$ holds.
\end{proof}

\begin{lemma}
Let $F\in \R[x,y,z]$ be irreducible over $\R$ but reducible over $\C$, and assume that $Z_\R(F)$ has dimension 2.
If $F$ does not satisfy $(i^*)$, then it satisfies $(ii)_\R$.
\end{lemma}
\begin{proof}
In this case, $F$ must be of the form $H \bar H$, with $H$ an irreducible complex polynomial. This is because
for every complex factor $H$ of $F$, $\bar H$ is also a factor, which follows from the
unique factorization property and the fact that $\bar F = F$. 
If we had a nontrivial factorization of the form $F=H_1\bar H_1H_2\bar H_2$, then $F = (H_1\bar H_1)\cdot (H_2\bar H_2)$ would be a factorization over $\R$, contradicting the assumption that
$F$ is irreducible over $\R$.

If property $(i^*)$ holds for either of the polynomials $H$, $\bar H$, then it holds for both, since 
\[\overline{Z_\C(H)\cap (A\times B\times C)} = Z_\C(\bar H) \cap (\bar A\times \bar B\times \bar C).\]
Since $Z_\C(F) = Z_\C(H)\cup Z_\C(\bar H)$, it follows that if $F$ does not satisfy property $(i^*)$, then at least one of $H$ or $\bar H$
does not satisfy it either, and thus neither of them does. 
Applying Theorem~\ref{thm:main2} to both $H$ and $\bar H$ gives that property $(ii)$ holds (over $\C$) for both.

Let $Y_H$, $Y_{\bar H}$ be the respective one-dimensional excluded subvarieties of $Z_\C(H)$, $Z_\C(\bar H)$ in the statement
of property $(ii)$. 
We put 
\[Y:=Y_H\cup Y_{\bar H}\cup Z_\C(H,\bar H),\]
and claim that $Z_\C(H,\bar H)$ is also at most one-dimensional. Indeed, $H=\bar H=0$ is equivalent to $\mathsf{Re}(H)=\mathsf{Im}(H)=0$, which is the common zero set of two coprime polynomials 
(that $\mathsf{Re}(H)$ and $\mathsf{Im}(H)$ are coprime follows from the irreducibility of $H$).
We thus conclude that, also in this case, property $(ii)$
holds for $F$ (over $\C$), with the excluded subvariety $Y$.

We can now complete the proof of this lemma exactly as in the proof of Lemma \ref{lem:realirred}, using this 
excluded subvariety $Y$.
\end{proof}

\section{Application: Collinear triples on complex algebraic curves}\label{sec:collinear}

In this section we use our main theorem to derive improvements of the results of 
Elekes and Szab\'o \cite{ES13} on collinear triples determined by a set of points on an algebraic curve.
We obtain stronger bounds that also hold in an unbalanced setting, and we extend them to the complex plane.
Our proof is similar to that of \cite{ES13}, in that it reduces the statements to Theorem \ref{thm:main2} 
in the same way that Elekes and Szab\'o reduced their statements in \cite{ES13} to the main theorem of \cite{ES12}.
But our proof is considerably simpler, partly because in \cite{ES13} the authors established a more general 
statement about real continuous curves, which they then applied to real algebraic curves, while we focus 
entirely on complex algebraic curves. The corresponding statements over $\R$ then follow directly.

The main result is the following theorem, from which the other results will be deduced.
Of course, just as in Theorem \ref{thm:main2}, an analogous bound holds for any permutation of $S_1,S_2,S_3$.
The explicit dependence of the constant of proportionality on the degree $d$ of the curves can easily be deduced from that in Theorem \ref{thm:main2}.
We call a triple $(p_1,p_2,p_3)$ of points \emph{proper} if no two of the points $p_i$ are the same, and we call it \emph{improper} otherwise.

\begin{theorem}\label{thm:threecurves}
Let $C_1,C_2,C_3$ be three (not necessarily distinct) irreducible algebraic curves of degree at most $d$ in $\C^2$,
and let $S_1\subset C_1, S_2\subset C_2,S_3\subset C_3$ be finite subsets.
Then the number of proper collinear triples in $S_1\times S_2\times S_3$ is
\[O_d\left(|S_1|^{1/2}|S_2|^{2/3}|S_3|^{2/3} + |S_1|^{1/2}\left(|S_1|^{1/2}+|S_2| + |S_3|\right)\right),\]
unless $C_1\cup C_2 \cup C_3$ is a line or a cubic curve.
\end{theorem}
\begin{proof}
Following \cite{ES13}, we say that \emph{collinearity is group-related} around a proper collinear triple $(p_1,p_2,p_3)\in C_1\times C_2\times C_3$, 
if for each $i=1,2,3$ there is an open subset $U_i$ of $C_i$ containing $p_i$, together with a one-to-one analytic map $\varphi_i:U_i\to \C$, 
such that a triple $(q_1,q_2,q_3)\in U_1\times U_2\times U_3$ is collinear if and only if 
 $$\varphi_1(q_1)+\varphi_2(q_2)+\varphi_3(q_3) = 0.$$
This situation is depicted in Figure \ref{fig:collinearity}.

\begin{figure}[htb]
\begin{center}
 \scalebox{0.5}{\includegraphics{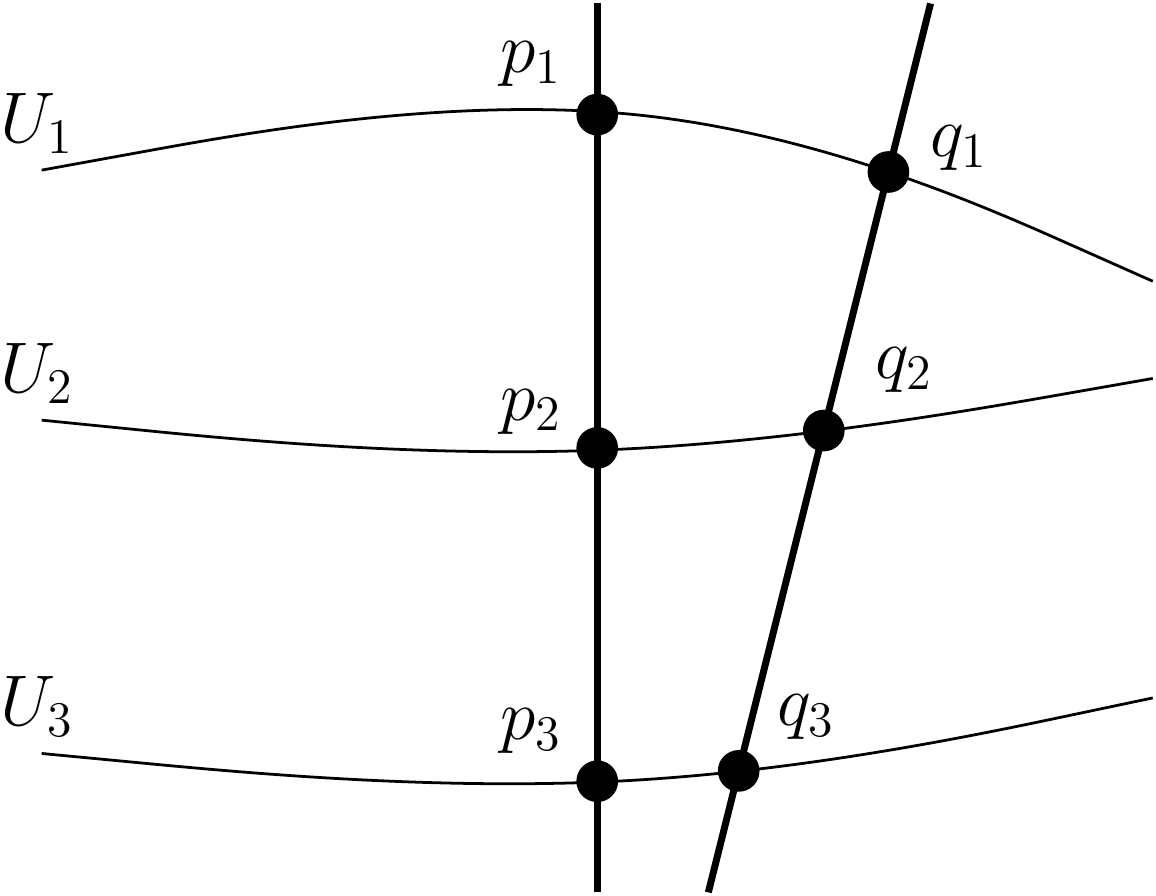}}
\caption{The situation in the definition of group-related collinearity.}
\label{fig:collinearity}
\end{center}
\end{figure}

We define a polynomial $L$ by
\[L(x_1,y_1,x_2,y_2,x_3,y_3) = \left|\begin{matrix} 1&x_1&y_1\\1&x_2&y_2\\1&x_3&y_3\end{matrix}\right|, \]
and a variety $X\subset \C^6$, equipped with coordinates $(x_1,y_1,x_2,y_2,x_3,y_3)$, by
\[X:= (C_1\times C_2\times C_3)\cap Z(L), \]
which is the set of all collinear triples in $C_1\times C_2\times C_3$.
Since $C_1\times C_2\times C_3$ is an irreducible three-dimensional variety, $X$ is purely two-dimensional by Lemma \ref{lem:dimension}, 
unless $L$ vanishes on all of $C_1\times C_2\times C_3$.
This exception only occurs if all triples in $C_1\times C_2\times C_3$ are collinear, which would 
imply that the three curves are the same line; this is excluded in the theorem.

By applying a generic rotation in $\C^2$ at the start of the proof, we can assume that no two points of $S_1$, $S_2$, or $S_3$ have the same $x$-coordinate.
Then the projection $\pi:\C^6\to\C^3$ defined by 
\[\pi(x_1,y_1,x_2,y_2,x_3,y_3)= (x_1,x_2,x_3)\]
is injective on the Cartesian product $S_1\times S_2\times S_3$.
Moreover, the image of $S_1\times S_2\times S_3$ is a Cartesian product $A_1\times A_2\times A_3$, with $A_i\subset \C$ of size $|A_i| = |S_i|$ for $i=1,2,3$.
Because of the generic rotation, $\cl(\pi(X))$ is also a purely two-dimensional variety.

The variety $X$ contains all collinear triples in $C_1\times C_2\times C_3$, but it also contains all improper triples. 
These are mapped onto the union of three planes (namely those defined by $x_1=x_2$, $x_2=x_3$, and $x_3=x_1$).
We remove these planes from $\cl(\pi(X))$, and we denote the closure of the remainder by $Y$.
If we write $M$ for the number of proper collinear triples in $S_1\times S_2\times S_3$,
then these $M$ triples are mapped to $M$ points in the intersection of $Y$ with the Cartesian product $A_1\times A_2\times A_3$.

Thus we can apply Theorem \ref{thm:main2} on each irreducible component of $Y$ to bound $M$,
noting that each such component is the zero set of some irreducible trivariate polynomial
(whose degree depends on $d$). This gives the bound in the statement of Theorem \ref{thm:threecurves}, 
unless condition $(ii)$ of Theorem \ref{thm:main2} holds on some irreducible component of $Y$.
Suppose $Y'$ is such a component, so condition $(ii)$ gives, for $i=1,2,3$, a number $t_i$, a 
neighborhood $D_i$ of $t_i$, and a one-to-one analytic map $\phi_i:D_i\to \C$, such that,
for each $(u,v,w)\in D_1\times D_2\times D_3$, $(u,v,w)\in Y'$ if and only if $\phi_1(u)+\phi_2(v)+\phi_3(w)=0$.
We can assume that $D_1\times D_2\times D_3$ does not contain any points that were added to $\pi(X)$ when taking the closure.

Write $\pi_i$ for the projection $(x_i,y_i)\mapsto x_i$.
We choose $p_i\in C_i$ so that $\pi_i(p_i) = t_i$.
We also pick an open neighborhood 
$U_i$ of $p_i$ in $C_i$ so that $\pi_i(U_i)\subset D_i$, 
and we define the analytic map $\varphi_i := \phi_i\circ \pi_i:U_i\to \C$.
By shrinking $U_i$, we can assume that $\varphi_i$ is one-to-one.
By shifting the triple $(p_1,p_2,p_3)$ slightly within $U_1\times U_2\times U_3$, we can assume 
that the triple is proper, since improper triples lie in a lower-dimensional subset of $Y'$.
Finally, we can assume that $p_1,p_2,p_3$ are regular points of their respective curves, since the
sets of singular points are discrete.
With these definitions, collinearity is group-related around some proper collinear triple $(p_1,p_2,p_3)\in C_1\times C_2\times C_3$, with $p_1,p_2,p_3$ regular.

We pick a point $q$ on $U_2$, close to $p_2$, and start to build a ``cantilever"\footnote{%
  The term ``cantilever'' was introduced for such a configuration in \cite{ES13}, reflecting 
  the fact that it resembles an overhanging structure supported on one end.} 
from $p_1$, $p_3$, and $q$ as in Figure \ref{fig:cantilever1}.
More precisely, we proceed as follows. 
Throughout, we assume that $q$ is chosen close enough to $p_2$ for each of the steps to work. 
For convenience, we modify the functions $\varphi_i$ so that $\varphi_1(p_1) = \varphi_2(p_2) = 0$ and $\varphi_2(q) = 1$ 
(the integer labels in Figure \ref{fig:cantilever1} show these values).
The line through $p_1$ and $q$ intersects $C_3$ in a point $r_3$ close to $p_3$, 
and since collinearity is group-related, we must have $\varphi_3(r_3) = -\varphi_1(p_1)-\varphi_2(q) = -1$.
Similarly, the line $p_3q$ intersects $C_1$ in a point $r_1$ close to $p_1$, with $\varphi_1(r_1) = -1$.
Next, the line $r_1r_3$ must intersect $C_2$ in a point $r_2$ close to $p_2$, with $\varphi_2(r_2) = 2$.
We can continue adding points this way as long as the $U_i$'s allow, 
and by choosing $q$ closer to $p_2$, we can continue for arbitrarily many steps.

\begin{figure}[htb]
 \begin{center}
\scalebox{0.4}{\includegraphics{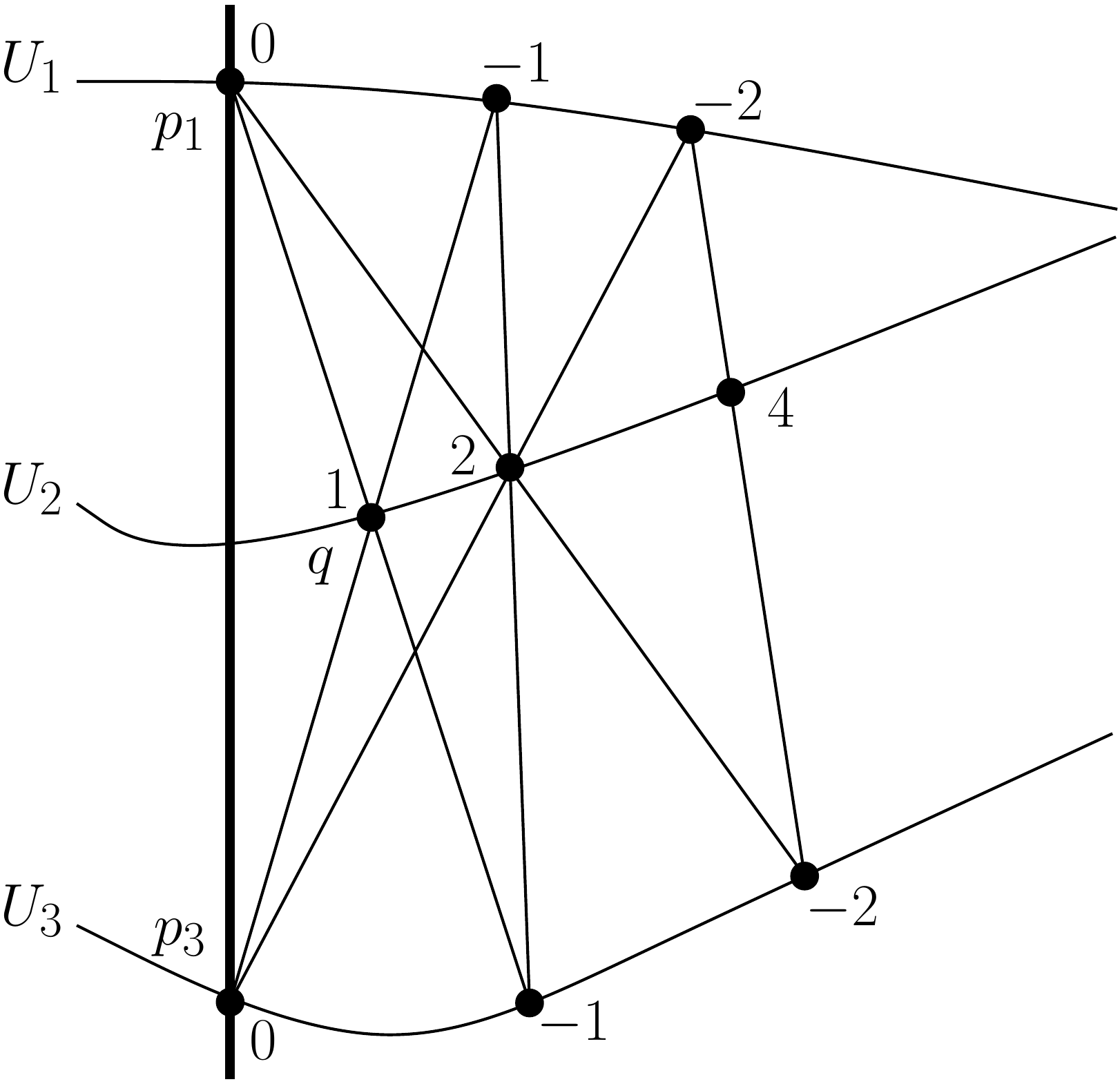}}
 \end{center}
 \caption{A cantilever of nine points built from $p_1$, $p_3$, and $q$.}
 \label{fig:cantilever1}
\end{figure}

The key observation is that, given the first nine points as depicted in Figure \ref{fig:cantilever1} (not counting $p_2$), the cantilever can be continued \emph{without using the curves} (see Figure \ref{fig:cantilever2}).
Specifically, given the points in Figure \ref{fig:cantilever1}, the point on $U_3$ with value $3$ is determined as the intersection point of two lines (which is not the case for the point with value $4$).
Next, the points on $U_1$ and $U_3$ with value $-3$ are determined, 
after which the point on $U_2$ with value $5$ is determined, etc.

\begin{figure}[htb]
\begin{center}
\scalebox{0.4}{\includegraphics{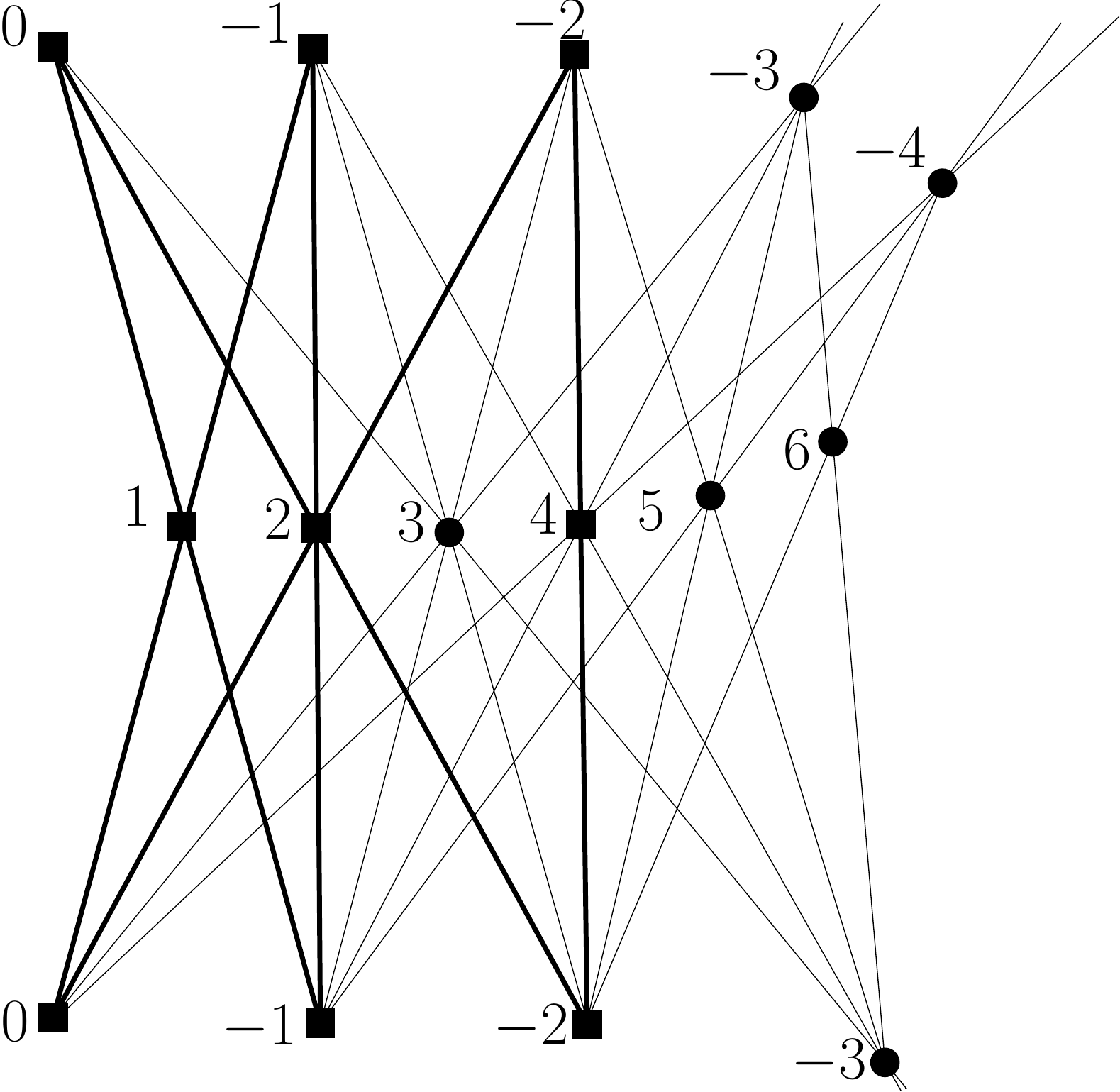}}
\end{center}
\caption{Given the nine points marked by squares, the other points are also determined.}
\label{fig:cantilever2}
\end{figure}

Given the curves $C_1,C_2,C_3$, the collinear triple $(p_1,p_2,p_3)$, and a choice of $q$ 
(close enough to $p_2$), we get a configuration of nine points on $C_1\cup C_2\cup C_3$ as in Figure \ref{fig:cantilever1}.
An easy and well-known fact is that there exists a cubic curve through any nine given points; 
let $D$ be a cubic through our configuration of nine points. As shown in Elekes and Szab\'o~\cite[Proposition 3.2]{ES13}, 
collinearity is group-related on any cubic curve, around any proper collinear triple of regular points 
(the proof consists of choosing a convenient representative of each type of cubic, with a coordinate system in which the group relation becomes apparent).
Therefore, any cantilever obtained by extending the cantilever of nine points on the cubic $D$ lies entirely in $D$.

Now choose $q$ so close to $p_2$ that we can build a cantilever that has at least $3d+1$ points in each $U_i$. 
For each $i=1,2,3$, these $3d+1$ points lie in the intersection $C_i\cap D$. Since their number is larger than
$\deg(C_i) \deg(D)$, B\'ezout's inequality (Theorem \ref{thm:bezout}) implies that $C_i$ and $D$ have a common factor.
If $D$ is irreducible then, since each $C_i$ is also irreducible, we conclude that $D=C_i$ for each $i$.
Hence $C_1\cup C_2\cup C_3$ is a cubic in this case, and we are done.

Suppose then that $D$ is the union of a conic $Q$ and a line $\ell$. If all the $C_i$'s are equal to $\ell$,
we get that $C_1\cup C_2\cup C_3$ is a line. If all the $C_i$'s are equal to $Q$, then
$ C_1\cup C_2\cup C_3$ is a conic, in which case $S_1$, $S_2$, $S_3$ cannot determine any proper collinearity.
We are thus left with the case where at least one $C_i$ is equal to $\ell$ and at least one is equal to $Q$, so in this case too
$C_1\cup C_2\cup C_3$ is a cubic, and the proof is complete. 
\end{proof}

\begin{corollary}\label{cor:baltrips}
Any $n$ points on an irreducible algebraic curve of degree $d$ in $\C^2$ determine
$O_d(n^{11/6})$ proper collinear triples, unless the curve is a line or a cubic.
\end{corollary}
\begin{proof}
Apply Theorem \ref{thm:threecurves} with $S_1=S_2=S_3$ equal to the given set of $n$ points, 
and $C_1=C_2=C_3$ equal to the given curve.
\end{proof}

Elekes and Szab\'o \cite{ES13} showed that when $C$ is any cubic curve, there are constructions that give $\Omega(n^2)$ proper collinear triples.
When $C$ is a line, there are of course $\Omega(n^3)$ proper collinear triples.
Thus excluding these curves is necessary in Corollary \ref{cor:baltrips}.
The same can be said for Theorem \ref{thm:threecurves} and the two corollaries below.

\begin{corollary}
Any $n$ points on an algebraic curve of degree $d$ in $\C^2$ determine $O_d(n^{11/6})$ proper collinear quadruples, unless the curve contains a line.
\end{corollary}
\begin{proof}
Suppose the curve determines $N$ proper collinear quadruples.
The curve has at most $d$ irreducible components, so by the pigeonhole principle, there must be 
four components (not necessarily distinct) that span $\Omega_d(N)$ proper collinear quadruples (with one point of the quadruple from each component).
Then every three among these span $\Omega_d(N)$ collinear triples, and thus, by Theorem \ref{thm:threecurves}, 
either $N = O_d(n^{11/6})$, or the union of these three components is a line or a cubic.
In the first case, we are done.
In the second case, we are also done, unless all three components are the same irreducible cubic.
If we have this for every three of the four components, then all four would be the same irreducible cubic, and they could not span any proper collinear quadruple at all, a contradiction.
\end{proof}

\begin{corollary}\label{cor:directions}
Any $n$ points on an irreducible algebraic curve of degree $d$ in $\C^2$ determine $\Omega(n^{4/3})$ distinct directions, unless the curve is a conic.
\end{corollary}
\begin{proof}
Let $C_2$ and $C_3$ be two copies of this curve, and let $C_1$ be the line at infinity in the projective plane.
We take both $S_2$ and $S_3$ to be the given set of $n$ points,
and we let $S_1$ be the set of all points  where a line through a pair of points from $S_2=S_3$ intersects $C_1$.
In other words, the points of $S_1$ correspond exactly to the directions determined by the given $n$ points.
Furthermore, the three curves $C_1,C_2,C_3$ determine $\binom{n}{2}$ proper collinear triples, one for each pair of distinct points from $S_2=S_3$.

Combining this lower bound with Theorem \ref{thm:threecurves} gives that either
\[n^2 = O_d\left(|S_1|^{1/2}\cdot n^{4/3}\right),\]
or $C_1\cup C_2\cup C_3$ is a line or cubic.
The former implies that the number of directions is $|S_1| = \Omega_d(n^{4/3})$.
In the latter case, 
the fact that $C_1$ is a line distinct from $C_2=C_3$ implies that $C_2=C_3$ is a conic.
\end{proof}

Finally, we note that the statements above also hold over $\R$.
Indeed, an irreducible algebraic curve in $\R^2$ is contained (under the standard embedding of $\R^2$ in $\C^2$) in an irreducible algebraic curve in $\C^2$, which has the same degree. 
A real line containing three points of the real curve is contained in a complex line, which contains at least three points of the complex curve.
Thus an upper bound on the number of proper collinear triples on a complex curve implies an upper bound on the number of proper collinear triples on a real curve.



\appendix

\section{Tools from algebraic geometry}
\label{sec:alggeom}

In this section we review some basic notions and facts from algebraic geometry, and establish various properties that we need in our proof.
Although the material reviewed here is fairly standard, we include it to make the paper more self-contained, and to aid readers whose background is mostly in combinatorial geometry.

\subsection{Definitions and basic facts}
A \emph{variety} in $\C^D$ is a set of the form 
\[ Z(f_1,\dots,f_k) 
:= \left\{(z_1,\ldots,z_D)\in \C^D\mid f_i(z_1,\ldots, z_D)=0 ~\text{for}~ i=1,\ldots,k\right\},\]
for polynomials $f_1,\ldots,f_k\in \C[z_1,\ldots,z_D]$.
Such sets are normally called \emph{affine} varieties, but since this is the only type of variety that we consider, we refer to them simply as varieties.
If $X, Y$ are varieties, then $X\cup Y$, $X\cap Y$ and $X\times Y$ are also varieties.

A \emph{subvariety} of a variety $X$ is a subset of $X$ that is a variety. 
A \emph{proper subvariety} of $X$ is a subvariety which is neither $X$ nor the empty set.
A variety is \emph{irreducible} if it is not the union of two proper subvarieties.
If $X$ and $Y$ are irreducible varieties, then the Cartesian product $X\times Y$ is also irreducible (Harris \cite[Exercise 5.9]{Ha92}).
An \emph{irreducible component} of a variety $X$ is an irreducible subvariety that is not a proper subvariety of any irreducible proper subvariety of $X$.
Every variety has a decomposition into finitely many irreducible components (see \cite[Theorem 5.7]{Ha92}).
This decomposition is unique (up to permutations), and any irreducible component of the variety must occur in it.

We are particularly interested in curves in $\C^2$. 
An \emph{(algebraic) curve} in $\C^2$ is any set of the form $Z(f)$ for $f\in \C[x,y]\minus\C$. 
The \emph{degree} of a curve $\gamma$ is the degree of a squarefree polynomial $f$ such that $\gamma=Z(f)$.
An irreducible component of a curve is a curve.
Note that by our definition, not every one-dimensional variety in $\C^2$ (see below) is a curve; 
for instance, $Z(xy, (y-1)y)$, i.e., the union of the $x$-axis and the point $(0,1)$, is a one-dimensional variety that cannot be described as the zero set of one polynomial.

\subsection{Dimension}\label{sec:dimension}

For a variety $X\subset \C^D$, we say that $z_0\in X$ is \emph{regular} if there exists a neighborhood $N$ of $z_0$ such that $X\cap N$ is a complex manifold (see \cite[Exercise 14.1]{Ha92}); we then say that the \emph{local dimension} of $X$ at $z_0$ is the dimension of that manifold (see \cite[Exercise 14.3]{Ha92}). Otherwise we say $z_0$ is \emph{singular}.

The \emph{dimension} $\dim(X)$ of an irreducible variety $X$ is the maximum of the local dimensions at its regular points. 
For a general variety $X$, 
$\dim(X)$ is the maximum of the dimensions of the irreducible components of $X$.
We refer to \cite[Lecture 11]{Ha92} for several equivalent definitions of dimension, but we note that we only use the simple properties stated below.

For varieties $X,Y$, 
we have $\dim(X\cup Y) = \max\{\dim(X),\dim(Y)\}$, 
and $\dim (X\times Y) = \dim(X) + \dim(Y)$.
We have $\dim(\C^D) = D$.
If $X\subset \C^D$ is a variety and $\dim(X) = D$, then $X = \C^D$.
If $f\in\C[z_1,\dots, z_D]\minus \C$, then $\dim(Z(f)) = D-1$.
If $f,g\in \C[z_1,\ldots,z_D]$ are coprime polynomials, and $Z(f,g)$ is nonempty, then $\dim(Z(f,g)) = D-2$.

We say that a variety is \emph{pure-dimensional} if it has the same local dimension at all of its regular points.
By our definition, curves are always pure-dimensional.
An irreducible variety is pure-dimensional (this is clear from the equivalent definitions in \cite[Lecture 11]{Ha92}),
 so a variety is pure-dimensional if and only if each of its irreducible components has the same dimension.
If $Y$ is an irreducible subvariety of 
$X$ and $\dim(X)=\dim(Y)$, 
then $Y$ is one of the irreducible components of $X$. 

A point on a variety in $\C^D$ is \emph{singular} if it is not regular.
The subset of singular points of an irreducible variety is a proper subvariety (see \cite[Exercise 14.3]{Ha92} or Hartshorne \cite[Theorem I.5.3]{Ha77}),
and therefore lower-dimensional. 
The same follows for reducible varieties: The set of singular points is the union of the sets of singular points of the irreducible components, 
and the set of intersection points of the various irreducible components; both sets have dimension lower than that of the component of highest dimension.

The following fact is very useful (see \cite[Exercise 11.6]{Ha92} and \cite[Proposition I.7.1]{Ha77}).
\begin{lemma}\label{lem:dimension}
Let $X\subset \C^D$ be an irreducible variety of dimension $k$ and $f\in \C[z_1,\ldots,z_D]$.
Then $X\cap Z(f)$ is either $X$, the empty set, or has pure dimension $k-1$.
\end{lemma}
It follows that, for a nonempty variety $X\subset \C^D$ 
(not necessarily irreducible) which is defined by $\ell$ polynomials, we have $\dim(X)\geq D-\ell$.


\subsection{Degree}\label{sec:degree}
We define the \emph{degree} $\deg(X)$ of an irreducible variety $X$ of dimension $k$ in $\C^D$ as in Heintz \cite{Hei}, by 
\[\deg(X) := \sup\{|X\cap L|
\mid L~\text{is a}~(D-k)\text{-flat}~\text{such that}~X\cap L~\text{is finite}
\}.\]
For a reducible variety $X$, we define $\deg(X)$ to be the sum of the degrees of its irreducible components.
For curves, this definition coincides with our earlier definition.
We have $\deg(\C^D)=1$ and $\deg(Z(f))= \deg(f)$ if $f$ is a square-free polynomial. 
If $S$ is a finite set, then $\deg(S) = |S|$.
If $X$ and $Y$ are varieties, then $\deg(X\times Y) = \deg(X)\cdot \deg(Y)$.
If $X\subset Y$ are pure-dimensional varieties of the same dimension, then $\deg(X)\leq \deg(Y)$.

The following bound is proved in Heintz \cite[Theorem 1]{Hei}. 

\begin{theorem}[\bf Generalized B\'ezout]\label{thm:bezout}
If $X$ and $Y$ are varieties in $\C^D$,
then 
\[\deg(X\cap Y) \leq \deg(X)\cdot \deg(Y).\]
\end{theorem}

Note that Theorem \ref{thm:bezout} implies B\'ezout's inequality for curves $C_1,C_2$ in the plane, which is usually stated in the following form: Either
$C_1\cap C_2$ is finite, and then $|C_1\cap C_2| = \deg(C_1\cap C_2)\leq \deg(C_1)\cdot \deg(C_2)$, or else $C_1\cap C_2$ is one-dimensional, which means that $C_1$ and $C_2$ have a common component.

The following bound is an immediate consequence. 
Since the degree of a reducible variety is the sum of the degrees of its irreducible components, this lemma also gives a bound on the number of irreducible components of a variety.

\begin{lemma}\label{lem:degreebound}
If a variety $X$ in $\C^D$ is defined by $m$ polynomials of degree $\delta_1,\ldots, \delta_m$, then 
\[\deg(X)\leq \prod_{i=1}^m \delta_i.\]
In particular, the number of irreducible components of $X$ is at most $\prod_{i=1}^m \delta_i$.
\end{lemma}

Finally, we state the Schwartz-Zippel lemma \cite{Sch80,Zi89}.
We state it over $\C$, although it holds over any field.

\begin{lemma}\label{lem:schwartzzippel}
Let $G\in\C[x_1,\ldots,x_D]$ be a nonzero polynomial, and $S\subset \C$ a finite set. Then
\[|Z(G)\cap S^D|\leq \deg(G)\cdot|S|^{D-1}. \]
\end{lemma}


\subsection{Projections and closure}
\label{sec:projclos}

We will frequently use standard projections,
i.e., projections that send a point of $\C^D$ to a point of $\C^{D'}$, for $D'<D$, by omitting some of its coordinates.
We will use the word \emph{projection} to refer to any such map, although most of the properties we state below are true for more general projections.

The image $\pi(X)$ of a variety $X$ under a projection  $\pi$ does \emph{not} have to be a variety; 
if we for instance apply $\pi:(x,y)\mapsto x$ 
 to $X = Z(xy-1)$, the image is $\pi(X) = \C\backslash\{0\}$.
Therefore, we will have to enlarge the image of a projection to make it a variety.

For any set $S\subset \C^D$, we define its \emph{(Zariski) closure} to be the intersection of all varieties containing $S$,
and denote it by $\cl(S)$.
This is the closure in the ``Zariski topology'' (see \cite[Lecture 2]{Ha92}); because we also deal with the standard topology of $\C^D$, we avoid the Zariski terminology, but we make an exception for ``closure''.
We need the following facts.
The closure of the union of finitely many sets equals the union of the closures, while the union of infinitely many sets \emph{contains} the union of the closure.
The closure of a finite product of sets is the product of the closures.
If $S\subset \C$ is infinite, then $\cl(S) = \C$.
The following lemma gives another useful fact.

\begin{lemma}\label{lem:dense}
Let $X$ be a pure-dimensional variety and $Y$ a lower-dimensional variety.
Then $$\cl(X\minus Y) = X.$$
\end{lemma}
\begin{proof}\hspace{-5pt}\footnote{We provide some short proofs in this subsection, because we could not find references for the exact statements that we need.}
Consider first the case where $X$ is irreducible. If $X\cap Y=\emptyset$ the assertion clearly holds, so we can assume this is not the case. 
Since $\dim(Y)<\dim (X)$, we have that $X\minus Y$ is nonempty, and $X\cap Y$ is a proper subvariety of $X$.
Suppose $\cl(X\minus Y) = Z$, for some proper subvariety $Z$ of $X$.
Then $X= (X\cap Y)\cup Z$, a union of two proper subvarieties, 
contradicting the assumption that $X$ is irreducible.

For the general case, let $X_1,\ldots, X_m$ be the irreducible components of $X$; so $\dim(X_i)=\dim(X)$ for each $i$, and $Y\neq X_i$ for any $i$.
We have $\cl(X_i\minus Y) = X_i$ for each $i$.
Thus $\cl(X\minus Y) = \bigcup_{i=1}^m \cl(X_i\minus Y) = \bigcup_{i=1}^m X_i = X$.
\end{proof}

We frequently combine a projection $\pi$ with the closure operation, giving us a variety $\cl(\pi(X))$, and we need the fact that this only requires adding a lower-dimensional set to $\pi(X)$.
We deduce this from the following theorem, a proof of which can be found in Basu, Pollack, and Roy \cite[Theorem 1.22]{BPR03},  \cite[Theorem 3.16]{Ha92}, or \cite[Exercise II.3.18]{Ha77}.  
A \emph{constructible} set in $\C^D$ is one that can be defined using
any boolean combination of polynomial equations, or, in an equivalent form that is more convenient to us, any set of the form
$\bigcup_{i=1}^s (X_i\minus Y_i)$
for varieties $X_i,Y_i$ in $\C^D$.

\begin{theorem}[{\bf Chevalley
}]\label{thm:chevalley}
Let $\pi:\C^D\to\C^{D'}$ be a projection and $X\subset \C^D$ a constructible set.
Then $\pi(X)$ is a constructible set.
\end{theorem}

The following lemma states that when taking the closure of the image of a projection, the set of added points is lower-dimensional.
For instance, when projecting a curve, the image is a curve with only finitely many points removed.

\begin{lemma}\label{lem:projectionclosure}
Let $\pi:\C^D\to\C^{D'}$ be a projection and $X\subset \C^D$ a variety.
Then
\[\cl(\pi(X))\minus \pi(X)\]
is contained in a lower-dimensional subvariety of $\cl(\pi(X))$ (and thus its closure is lower-dimensional).
\end{lemma}
\begin{proof}
By Theorem \ref{thm:chevalley}, $\pi(X)$ is a constructible set, 
so we can write $\pi(X) = \bigcup_{i=1}^s (X_i\minus Y_i)$
with varieties $X_i,Y_i$ in $\C^{D'}$.
More precisely, we can write
$\pi(X) = \bigcup_{i=1}^s (X_i\minus (Y_i\cap X_i))$, assume that each $X_i$ is irreducible, and assume that $Y_i\cap X_i$ is a proper subvariety of $X_i$ for each $i$.
By Lemma \ref{lem:dense}, we have $\cl(\pi(X)) = \bigcup_{i=1}^s X_i$, and thus
\[
\cl(\pi(X))\minus \pi(X) \subseteq \bigcup_{i=1}^s(Y_i\cap X_i). 
\]
Since $Y_i\cap X_i$ is a proper subvariety of $X_i$, there must be some polynomial that vanishes on $Y_i$ but not on $X_i$. Lemma \ref{lem:dimension} then implies that $Y_i\cap X_i$ is lower-dimensional, which proves the lemma.
\end{proof}

Finally, we need the basic fact that a projection, combined with closure, does not increase the dimension or degree of a variety.

\begin{lemma}\label{lem:projectionpreserves}
Let $\pi:\C^D\to \C^{D'}$ be a projection and $X\subset \C^D$ a variety.
Then
\[\dim(\cl(\pi(X)))\leq \dim(X)~~~\text{and}~~~\deg(\cl(\pi(X)))\leq \deg(X).\]
\end{lemma}
\begin{proof}
The first inequality follows from \cite[Theorem 11.12]{Ha92}.
The second inequality is implied by the following facts.
By Theorem \ref{thm:chevalley}, $\pi(X)$ is constructible. By Heintz \cite[Remark 2]{Hei}, the degree of a constructible set equals the degree of its closure. 
By \cite[Lemma 2]{Hei}, $\deg(\varphi(X))\leq \deg(X)$ for any linear map $\varphi$. 
\end{proof}

\subsection{Analysis}\label{sec:analysis}

We require several concepts and facts from complex analysis in several variables.\footnote{We have placed these tools from analysis in the ``Tools from algebraic geometry'' section only for convenience.}
Precise definitions related to \emph{analytic} (or equivalently, \emph{holomorphic}) maps from $\C^m$ to $\C^n$ can be found in Fritzsche and Grauert \cite{FG02}.
In brief,
a function $f$ from an open set in $\C^m$ to $\C$ is analytic if it is locally represented by a power series in the $m$ coordinate variables, 
and a map $(f_1,\ldots, f_n)$ from an open set in $\C^m$ to $\C^n$ is analytic if each $f_i$ is analytic.
Polynomial maps are of course analytic.
A map $f$ between two open sets in $\C^m$ is \emph{bianalytic} if it is analytic and bijective, and $f^{-1}$  is analytic.

For an analytic map ${\bf f} = (f_1,\ldots,f_m)$ from an open set $U\subset \C^k$ to $\C^m$,
with $k\ge m$, we define the \emph{Jacobian matrix} of ${\bf f}$ at $z = (z_1,\ldots,z_k)\in U$ to be 
\[
J_{\bf f}(z) := \left(\frac{\bd f_i}{ \bd z_j}(z) \right)_{\substack{1\leq i\leq m\\1\leq j\leq k}}.
\]
Given a representation $\C^k=\C^n\times \C^m$, we write
\[
J_{\bf f}^{m}(z) := \left(\frac{\bd f_i}{ \bd z_j}(z) \right)_{\substack{1\leq i\leq m\\n+1\leq j\leq n+m}}
\]
for the rightmost square submatrix of $J_{\bf f}(z)$.

For proofs of Lemmas \ref{lem:ift1}, \ref{lem:ift2}, and \ref{lem:imt} below, see Fritzsche and Grauert \cite[Chapter 7]{FG02}. The first lemma implies the other two.

\begin{lemma}[{\bf Implicit function theorem}]\label{lem:ift1}
Let $B\subset \C^n\times \C^m$ be an open set, ${\bf f}:B\to \C^m$ an analytic map,
and $z_0\in B$ a point with ${\bf f}(z_0) = 0$ and $\det J_{\bf f}^{m}(z_0) \neq 0$.
Then there are open sets $U\subset \C^n,V\subset \C^m$ such that $z_0\in U\times V\subset B$, and an analytic map ${\bf g}:U\to V$ such that 
\[
\{(u,v)\in U\times V\mid {\bf f}(u,v)=0\}
= \{ (u,{\bf g}(u))\mid u\in U\}.
\]
\end{lemma}

One instance of the implicit function theorem is especially important to us, and we state it here explicitly for convenience.

\begin{lemma}[{\bf Implicit function theorem for a surface in $\C^3$}]\label{lem:ift2}
Let $B\subset \C^3$ be an open set, $f\in \C[z_1,z_2,z_3]\backslash\{0\}$ a polynomial,
and $(a,b,c)\in B$ a point with $f(a,b,c) = 0$ and 
$\bd f/\bd z_3(a,b,c) \neq 0$.
Then there are open sets $U\subset \C^2, V\subset \C$ such that $(a,b,c)\in U\times V\subset B$, 
and an analytic map $g:U\to V$ such that 
\[\{(z_1,z_2,z_3)\in U\times V\mid f(z_1,z_2,z_3)=0\}
= \{ (z_1,z_2,g(z_1,z_2))\mid (z_1,z_2)\in U\}.\]
\end{lemma}

In certain parts of our proof we use the following equivalent formulation of the implicit function theorem (see \cite{FG02}).
We write ${\bf f}|_{U}$ for the restriction of the function ${\bf f}$ to  a subset $U$ of its domain.

\begin{lemma}[{\bf Inverse mapping theorem}]\label{lem:imt}
Let $B_1,B_2\subset \C^n$ be open, $z_0\in B_1$, and ${\bf f}:B_1\to B_2$ analytic.
Then $\det J_{\bf f}(z_0)\neq 0$ if and only if
there are open sets $U\subset B_1$, $V\subset B_2$ such that $z_0\in U$, ${\bf f}(z_0)\in V$, 
and ${\bf f}|_{U}:U\to V$ is bianalytic.
\end{lemma}

\end{document}